\documentclass[12pt,a4paper]{amsart}
\usepackage[top=35mm, bottom=38mm, left=30mm, right=30mm]{geometry}
\usepackage[hidelinks]{hyperref}
\usepackage{mathptmx}
\usepackage{eucal}
\usepackage{graphicx}
\usepackage{mathrsfs}
\usepackage{amssymb}
\usepackage{amsmath}
\usepackage{amsthm}
\usepackage{amscd}

\newtheorem{thm}{Theorem}[section]
\newtheorem*{vdw-thm}{Van der Waerden Theorem}
\newtheorem*{tmr-thm}{Topological Multiple Recurrence Theorem}
\newtheorem*{tmr-thm-ii}{Topological Multiple Recurrence Theorem II}
\newtheorem*{sz-thm}{Szemer\'{e}di Theorem}
\newtheorem*{mr-thm}{Multiple Recurrence Theorem}
\newtheorem*{cs-thm}{Central Sets Theorem}

\newtheorem{cor}[thm]{Corollary}
\newtheorem{lem}[thm]{Lemma}
\newtheorem{prop}[thm]{Proposition}

\newtheorem{ques}[thm]{Question}

\theoremstyle{definition}
\newtheorem{defn}[thm]{Definition}
\newtheorem{rem}[thm]{Remark}
\newtheorem{exmp}[thm]{Example}

\numberwithin{equation}{section}

\DeclareMathOperator{\Orb}{Orb}
\DeclareMathOperator{\diam}{diam}
\DeclareMathOperator{\Tran}{Tran}
\DeclareMathOperator{\FS}{FS}
\DeclareMathOperator{\supp}{supp}

\newcommand{\N}{\mathbb{N}}
\newcommand{\AP}{\mathcal{AP}}
\newcommand{\F}{\mathcal{F}}
\newcommand{\Zp}{{\mathbb{Z}_+}}
\newcommand{\Z}{\mathbb{Z}}
\newcommand{\eps}{\varepsilon}
\newcommand{\set}[1]{\{#1\}}

\begin{document}
\title{Multi-recurrence and van der Waerden systems}
\author{Dominik Kwietniak}
\address[D. Kwietniak]{Institute of Mathematics, Jagiellonian University in Krak\'{o}w, ul. {\L}ojasiewicza 6, 30-348 Krak\'{o}w, Poland}
   \email{dominik.kwietniak@gmail.com}
\author{Jian Li}
\address[J. Li]{Department of Mathematics, Shantou University, Shantou, Guangdong 515063, P.R. China}
\thanks{Corresponding author: Jian Li (lijian09@mail.ustc.edu.cn)}
\email{lijian09@mail.ustc.edu.cn}
\author{Piotr Oprocha}
\address[P. Oprocha]{AGH University of Science and Technology, Faculty of Applied
	Mathematics, al.
	Mickiewicza 30, 30-059 Krak\'ow, Poland
	-- and --
	National Supercomputing Centre IT4Innovations, Division of the University of Ostrava,
	Institute for Research and Applications of Fuzzy Modeling,
	30. dubna 22, 70103 Ostrava,
	Czech Republic}
\email{oprocha@agh.edu.pl}
\author{Xiangdong Ye}
\address[X. Ye]{Wu Wen-Tsun Key Laboratory of Mathematics, USTC, Chinese Academy of Sciences,
Department of Mathematics, University of Science and Technology of China,
Hefei, Anhui, 230026, P.R. China}
\email{yexd@ustc.edu.cn}

\begin{abstract} We explore recurrence properties arising from dynamical approach
to the van der Waerden Theorem and similar combinatorial problems. We describe relations between
these properties and study their consequences for dynamics. In particular, we present a measure-theoretical analog of a result of
Glasner on multi-transitivity of topologically weakly mixing minimal maps. We also obtain a dynamical proof of
the existence of a $C$-set with zero Banach density.
\end{abstract}
\date{\today}
\maketitle

\section{Introduction}

We study multiple-recurrence properties of dynamical systems on comp\-act met\-ric spac\-es. We use topological dynamics to characterize selected classes of subsets of $\N$ (e.g. IP-sets, C-sets, etc.) and to gain a better understanding of some
classes of transitive systems. The idea goes back to the work of Furstenberg in the 1970s.

Our starting point is the following result published in \cite{W27}.
\begin{vdw-thm}   \label{thm:APset}
If $\N$ is partitioned into finitely many 
subsets, then
one of these sets contains arithmetic progressions of arbitrary finite
length.
\end{vdw-thm}

In 1978, Furstenberg and Weiss \cite{FW78} obtained a dynamical proof of the van der Waerden Theorem.
They proved the Topological Multiple Recurrence Theorem and showed that it is equivalent to the van der Waerden Theorem.
``Equivalent'' means here that any of these results may be proved by assuming the other is true.

\begin{tmr-thm}
Let $(X,T)$ be a compact dynamical system. 
Then there exists a point $x\in X$ such that for any $d\in\mathbb{N}$ there is a strictly increasing sequence
$\{n_k\}_{k=1}^\infty$ in $\N$ with
$T^{i n_k} x\to x$ as $k\to\infty$ for every $i=1,2,\dotsc,d$.
\end{tmr-thm}

We call a point $x\in X$ fulfilling the conclusion of the topological multiple recurrence theorem
a \emph{multi-recurrent point}. In Section 3 we show that the set  of all multi-recurrent points is a $G_\delta$ subset of $X$; it is a
residual set if $(X,T)$ is minimal; and when $(X,T)$ is distal or uniformly rigid, then every point is multi-recurrent.
We also provide an example of a substitution subshift with minimal points which are not multi-recurrent. Then we prove
that multi-recurrent points can be lifted through a distal extension but this does not need to hold for a proximal extension
(we strongly believe that it can not be lifted by weakly mixing extension, but we do not have an example at this moment).
Using ergodic theory we show that the collection of multi-recurrent points which return to any of their
neighborhoods with positive upper density
has full measure for every invariant measure. If the invariant measure is weakly mixing and fully
supported then for almost every $x\in X$ and every $d\geq 1$ the diagonal $d$-tuple $(x,x,\ldots, x)$
has a dense orbit under the action of $T\times T^2\times\dotsb\times T^d$, which can be viewed as
a measure-theoretical version of a result of Glasner on topological weakly mixing minimal maps \cite{G94}.

Let us mention another equivalent version of the Topological Multiple Recurrence Theorem which
shows the relationship between these results and Furstenberg's Multiple Recurrence Theorem
for measure preserving systems (the so-called ``ergodic
Szemer\'{e}di Theorem''). It also comes from \cite[Theorem 1.5]{FW78}.
For a short and elegant proof see \cite[Theorem 1.56]{G}.

\begin{tmr-thm-ii}  \label{thm:1.3}
If a dynamical system $(X,T)$ is minimal, then for any $d\in\mathbb{N}$ and any non-empty open subset $U$ of $X$,
there exists a positive integer $n\geq 1$ with
\[U\cap T^{-n}U\cap T^{-2n}U\cap \dotsb\cap T^{-dn}U\neq\emptyset.\]
\end{tmr-thm-ii}

Inspired by this result, we introduce a new class of dynamical systems, which we call \emph{van der Waerden systems}, that is systems $(X,T)$ such that for every non-empty open subset $U$ of $X$
and for every $d\in\N$ there exists an $n\in\N$ such that
\[U\cap T^{-n}U\cap T^{-2n}U\cap \dotsb \cap T^{-dn}U\neq\emptyset\]
and we will study their basic properties in Section 4.
By the second variant of Topological Multiple Recurrence Theorem every minimal system is a van der Waerden system
and it is also not hard to see that $(X,T)$ is a  van der Waerden system
if and only if its multi-recurrent points are dense in $X$.

A generalization of van der Waerden Theorem is Szemer\'edi's Theorem \cite{S75}, proved in 1975.
\begin{sz-thm}
If $F\subset\N$ has positive upper density, then it contains arithmetic progressions of arbitrary finite
length.
\end{sz-thm}
Two years later, in 1977,  Furstenberg presented a new proof of
Szemer\'edi Theorem using dynamical systems approach. Furstenberg's proof is based on the equivalence of Szemer{\'e}di Theorem and the following Multiple Recurrence Theorem (see \cite{F77}).

\begin{mr-thm}  \label{thm:F-MRT}
If $(X,\mathcal{B},\mu)$ is a probability space and
$T$ is a measure preserving transformation of $(X,\mathcal{B},\mu)$,
then for any $d\in\N$ and any set $A\in\mathcal{B}$ with $\mu(A)>0$, there exists an integers $n\geq 1$ with
\[\mu(A\cap T^{-n}A\cap \dotsb T^{-2n}A\cap \dotsb\cap T^{-dn}A)>0.\]
\end{mr-thm}

It follows that every compact dynamical system with a fully supported invariant measure is a
van der Waerden system. We examine whether the converse is true. It turns out that there exists a topologically
strongly mixing system which is a van der Waerden system, but the only
invariant measure is a point mass on a fixed point, see Remark \ref{rem:5.6}. We also provide an example of a strongly
mixing system which is not a van der Waerden system.

While we were preparing this paper we found a work of Host et al. \cite{Kra15} which studies closely related problems,
but from a different point of view which emphasises the connection between recurrence properties and associated \emph{sets of (multiple) recurrence} (see
\cite[Definitions 2.1 \& 2.9]{Kra15}). Here we focus on recurrence of a single point in a concrete dynamical systems, and this complements the approach of \cite{Kra15}.


Our study of van der Waerden systems leads naturally to $\AP$-recurrent points.
We say that a point $x$ is \textit{$\AP$-recurrent} if for every neighborhood $U$ of $x$ the set of return times of $x$ to $U$ contains  
arithmetic progressions of arbitrary finite length.
It is clear that every multi-recurrent point is $\AP$-recurrent, but the converse is not true.
It is a consequence of the following characterization: a point is $\AP$-recurrent if and only if the closure of its orbit is
a van der Waerden system. A nice property of $\AP$-recurrent points is that they can be lifted through factor maps.

In~\cite{F81}, Furstenberg defined central subsets of $\N$ in terms
of some notions from topological dynamics. He showed that any finite
partition of $\N$ must contain a central set in one of its cells and proved the following Central
Sets Theorem~\cite[Proposition 8.21]{F81}.

\begin{cs-thm}
Let $C$ be a central set of $\N$.
Let $d\in\N$ and for each $i\in\{1,2,\dotsc,d\}$, let $\{p_n^{(i)}\}_{n=1}^\infty$ be a sequence in $\Z$.
Then there exist a sequence  $\{a_n\}_{n=1}^\infty$ in $\N$ and a sequence $\{H_n\}_{n=1}^\infty$ of finite subsets of $\N$
such that
\begin{enumerate}
  \item for every $n\in\N$, $\max H_n<\min H_{n+1}$ and
  \item for every  finite subset $F$ of $\N$ and every $i\in\{1,2,\dotsc,d\}$,
  \[\sum_{n\in F}\Bigl(a_n+\sum_{j\in H_n} p_j^{(i)}\Bigr)\in C.\]
\end{enumerate}
\end{cs-thm}
Central Sets Theorem has very strong combinatorial consequences, such as Rado's Theorem \cite{R33}.
The authors in~\cite{DHS08} proved a stronger version of the Central Sets Theorem
valid for an arbitrary semigroup $S$ and proposed to call a subset of $S$ a C-set
if it satisfies the conclusion of this version of the Central Sets Theorem.
A dynamical characterization of C-sets was obtained in~\cite{Li2012} by introducing a class of
dynamical systems satisfying the multiple IP-recurrence property. Note that C-sets considered in~\cite{Li2012} are subsets of $\Z$, however
Neil Hindman pointed out to the second author of this paper
that a similar characterization also holds for C-sets in $\N$.
\footnote{See also the review of \cite{Li2012} by N. Hindman in MathSciNet, MR2890544.}
A dynamical characterization of C-sets in an arbitrary semigroup $S$ is provided in \cite{J11} .

We study the multiple IP-recurrence property in
Section~5. We show that every transitive system with the multiple IP-recurrence property is either equicontinuous or sensitive.
This result generalizes theorems of Akin, Auslander and Berg \cite{AAB} and Glasner and Weiss \cite{GW93}.
We also provide an example of a strongly mixing system which is a van der Waerden system
but does not have the multiple IP-recurrence property.
We characterize bounded density shifts with the multiple IP-recurrent property. Combining this result
with the dynamical characterization of $C$-sets we obtain a dynamical proof of the main result of \cite{H09}:
there is a $C$-set in $\N$ with zero Banach density.

As seen above, the notion of a multi-recurrent point, which is parallel to the notion of a recurrent point
provides some insight to the theory of dynamical systems.
In the same spirit we define the notion of a multi-non-wandering point parallel to the classical
notion of a non-wandering point.
In section 6, we study the relations between multi-non-wandering points and the
sets containing arithmetic progressions of arbitrary finite length.
In particular, we provide a link between multi-non-wandering sets and $\AP$-recurrence.

By what we said above, it is easy to see that a transitive van der Waerden system
can be viewed as a generalization of an $E$-system (transitive system with a full supported invariant measure).
In a transitive van der Waerden system
each transitive point is $\AP$-recurrent,
and the set of multi-recurrent points is dense.
Note that for an $E$-system, the return time set of a transitive point to its neighborhood
has positive upper Banach density and at the same time,
the set of recurrent points with positive lower density of return time sets is dense.
For an $M$-system (transitive system with a dense set of minimal points),
this can be explained using piecewise syndetic sets and syndetic sets.

\section{Preliminaries}

In this section, we present basic notations, definitions and results.
\subsection{Subsets of positive integers}
Denote by $\mathbb{N}$ ($\Zp$ and $\Z$, respectively) the set of all positive integers
(nonnegative integers and integers, respectively).

A \emph{Furstenberg family} or simply a \emph{family} on $\N$ is any collection $\F$ of subsets of $\N$
which is hereditary upwards, i.e.
if $A\in \F$ and $A\subset B\subset \N$ then $B\in \F$. A \emph{dual family} for $\F$, denoted by $\F^*$, consists of
sets that meet every element of $\F$, i.e. $A\in \F^*$ provided that $\N\setminus A\not\in \F$. Clearly, $\F^{**}=\F$.

Given a sequence $\{p_i\}_{i=1}^\infty$ in $\N$, define the set of finite sums of $\{p_i\}_{i=1}^\infty$ as
\[\FS\{p_i\}_{i=1}^\infty =\Bigl\{\sum_{i\in\alpha} p_i\colon \alpha\text{ is a non-empty finite subset of }\N\Bigr\}.\]
We say that a subset $F$ of $\N$ is
\begin{enumerate}
\item an \emph{IP-set} if there exists a sequence $\{p_i\}_{i=1}^\infty\subset \N$
such that $FS\{p_i\}_{i=1}^\infty\subset F$;
\item an \emph{AP-set} if it contains arbitrarily long arithmetic progressions,
that is, for every $d\geq 1$, there are $a,n\in\N$ such that $\{a,a+n,\dotsc,a+dn\}\subset F$. The family of all AP-sets is denoted by $\AP$;
\item \emph{thick} if it contains arbitrarily long blocks of consecutive integers, that is,
for every $d\geq 1$ there is $n\in\N$ such that $\{n,n+1,\dotsc,n+d\}\subset F$;
\item \emph{syndetic} if it has bounded gaps, that is, for some $N\in\N$ and every $k\in\N$ we have
$\{k,k+1,\dotsc,k+N\}\cap F\neq\emptyset$;
\item \emph{co-finite} it it has finite complement, i.e. $\N\setminus F$ is finite.
\item an \emph{IP$^*$-set} (\emph{AP$^*$-set}, respectively) if it has non-empty intersection with every IP-set
(AP-set, respectively), that is it belongs to an appropriate dual family.
\end{enumerate}

It is easy to see that a subset $F$ of $\N$ is syndetic
if and only if it has non-empty intersection with every thick set, i.e. is in the family dual to all thick sets.
Every thick set is an IP-set, hence every IP$^*$-set is syndetic.

A family $\F$ has the \emph{Ramsey property}
if $F\in\F$ and $F=F_1\cup F_2$ imply that $F_i\in \F$ for some $i\in\{1,2\}$.
It is not hard to see that the van der Waerden theorem is equivalent to the fact that
 the  family $\AP$ has the Ramsey property.

Let $F$ be a subset of $\mathbb{Z}_+$. Define the \emph{upper density} $\overline{d}(F)$ of $F$ by
\[ \overline{d}(F)=\limsup_{n\to\infty} \frac{\#(F\cap[0,n-1])}{n},\]
where $\#(\cdot)$ is the number of elements of a set.
Similarly, $\underline{d}(F)$, the \emph{lower density} of $F$, is defined by
\[ \underline{d}(F)=\liminf_{n\to\infty} \frac{\#(F\cap[0,n-1])}{n}.\]
The \emph{upper Banach density} $BD^*(F)$ and  \emph{lower Banach density} $BD_*(F)$ are defined by
\[
BD^*(F)=\limsup_{N-M\to\infty} \frac{\#(F\cap[M,N])}{N-M+1}, \quad BD_*(F)
=\liminf_{N-M\to\infty} \frac{\#(F\cap[M,N])}{N-M+1}.
\]

\subsection{Topological dynamics}
By a \emph{(topological) dynamical system} we mean a pair $(X,T)$
consisting of a compact metric space $(X,\rho)$ and a continuous map $T\colon X\to X$.
If $X$ is a singleton, then we say that $(X,T)$ is \emph{trivial}.
If $K\subset X$ is a nonempty closed subset satisfying $T(K)\subset K$,
then we say that $(K,T)$ is a \emph{subsystem} of $(X,T)$ and $(X,T)$ is \emph{minimal} if it has no proper subsystems.
The \emph{(positive) orbit of $x$ under $T$} is the set $\Orb(x,T)=\{T^n x\colon n\in\Zp\}$.
Clearly, $\bigl(\overline{\Orb(x,T)},T\bigr)$ is a subsystem of $(X,T)$ and $(X,T)$ is minimal if $\overline{\Orb(x,T)}=X$
for every $x\in X$.

We say that a point $x\in X$ is
\begin{enumerate}
\item \emph{minimal}, if $x$ belongs to some minimal subsystem of $(X,T)$;
\item \emph{recurrent}, if $\liminf_{n\to \infty}\rho(T^{n}x,x)=0$;
\item \emph{transitive}, if $\overline{\Orb(x,T)}=X$.
\end{enumerate}
For a point $x\in X$ and subsets $U,V\subset X$,
we define the following sets of \emph{transfer times}:
\begin{eqnarray*}
N(U,V)&=&\{n\in\N\colon T^nU\cap V\neq\emptyset\}=\{n\in\N\colon U\cap T^{-n}V\neq\emptyset\},\\
N(x,U)&=&\{n\in\N\colon T^nx\in U\}.
\end{eqnarray*}
To emphasize that we are calculating the above sets using transformation $T$
we will sometimes write $N_T(x,U)$ and $N_T(U,V)$.

We say that a dynamical system $(X,T)$ is
\begin{enumerate}
  \item \emph{transitive} if $N(U,V)\neq \emptyset$ for every
two non-empty open subsets $U$ and $V$ of $X$;
\item \emph{totally transitive} if $(X,T^n)$ is transitive for every $n\in\mathbb{N}$;
\item \emph{(topologically) weakly mixing} if the product system $(X\times X,T\times T)$ is transitive;
\item \emph{(topologically) strongly mixing} if for every
two non-empty open subsets $U$ and $V$ of $X$, the set of transfer times $N(U,V)$ is cofinite.
\end{enumerate}

Denote by  $\Tran(X,T)$
the set of all transitive points of $(X,T)$.
It is easy to see that if a dynamical system $(X,T)$ is transitive then
$\Tran(X,T)$ is a dense $G_\delta$ subset of $X$.
It is also clear that a dynamical system $(X,T)$ is minimal if and only if $\Tran(X,T)=X$,
and a point $x\in X$ is minimal if and only if $\bigl(\overline{\Orb(x,T)},T\bigr)$ is a minimal system.

The following characterizations of recurrent points and minimal points are well-known (see, e.g., \cite{F81}).
\begin{lem}
Let $(X,T)$ be a dynamical system. A point $x\in X$ is
\begin{enumerate}
  \item recurrent if and only if for every open neighborhood $U$ of $x$
  the set $N(x,U)$ contains an IP-set;
  \item minimal if and only if for every open neighborhood $U$ of $x$
  the set $N(x,U)$ is syndetic.
\end{enumerate}
\end{lem}

A dynamical system $(X,T)$ is \emph{equicontinuous}
if for every $\varepsilon>0$ there is a $\delta>0$ such that
if $x,y\in X$ with $\rho(x,y)<\delta$ then
$\rho(T^nx,T^ny)<\varepsilon$ for $n=0,1,2,\dotsc$. A point $x\in X$ is \emph{equicontinuous}
if for every $\varepsilon>0$ there is a $\delta>0$ such that for every $y\in X$ with $\rho(x,y)<\delta$,
$\rho(T^nx,T^ny)<\varepsilon$ for all $n\in\Zp$.
By compactness, $(X,T)$ is equicontinuous if and only if every point in $X$ is equicontinuous.

We say that a dynamical system $(X,T)$ has \emph{sensitive dependence on initial condition} or
briefly $(X,T)$ is \emph{sensitive}
if there exists a $\delta>0$ such that for every $x\in X$ and every neighborhood $U$ of $x$
there exist $y\in U$ and $n\in\N$ such that $\rho(T^nx,T^ny)>\delta$.

A transitive system is \emph{almost equicontinuous} if there is at least one equicontinuous point.
It is known that if $(X,T)$ is almost equicontinuous then the set of equicontinuous points coincides
with the set of all transitive points and additionally $(X,T)$ is \emph{uniformly rigid},
that is for every $\varepsilon>0$ there exists an $n\in\N$ such that $\rho(T^nx,x)<\varepsilon$ for all $x\in X$.
We also have the following dichotomy:
if a dynamical system $(X,T)$ is transitive, then it is either almost equicontinuous or sensitive.
See \cite{AAB,GW93} for proofs and more details.

A pair $(x,y)\in X^2$ is \emph{proximal} if $\liminf_{n\to\infty}\rho(T^nx,T^ny)=0$,
and \emph{distal} if it is not proximal, that is $\liminf_{n\to\infty}\rho(T^nx,T^ny)>0$.
A point $x$ is \emph{distal} if $(x,y)$ is distal for any $y\in \overline{\Orb(x,T)}$ with $y\neq x$.
If every point in $X$ is distal then we say that $(X,T)$ is \emph{distal}.

Let $(X,T)$ and $(Y,S)$ be two dynamical systems. If there is a continuous surjection
$\pi: X \to Y$ with $\pi\circ T = S\circ \pi$,
then we say that $\pi$ is a \emph{factor map}, the system
$(Y,S)$ is a \emph{factor} of $(X,T)$ or $(X, T)$ is an \emph{extension} of $(Y,S)$.

A factor map $\pi: X \to Y$ is:
\begin{enumerate}
\item \emph{proximal} if $(x_1,x_2)\in X^2$ is proximal provided $\pi(x_1)=\pi(x_2)$;
\item \emph{distal} if  $(x_1,x_2)\in X^2$ is distal provided $\pi(x_1)=\pi(x_2)$ with $x_1\neq x_2$;
\item \emph{almost one-to-one} if there exists a residual subset $G$ of $X$ such that
$\pi^{-1}(\pi(x))=\{x\}$ for any $x\in G$.
\end{enumerate}

Let $M(X)$ be the set of Borel probability measures on $X$.
We are interested in those members of $M(X)$ that are invariant measures for $T$. Therefore, denote by $M(X,T)$
the set consisting of all
$\mu\in M(X)$ making $T$ a measure-preserving transformation of $(X,\mathcal{B}(X),\mu)$,
where $\mathcal{B}(X)$ is the Borel $\sigma$-algebra of $X$.
By the Krylov-Bogolyubov Theorem, $M(X,T)$ is nonempty.

The \emph{support} of a measure $\mu\in M(X)$, denoted by $\supp(\mu)$,
is the smallest closed subset $C$ of $X$ such that $\mu(C)=1$.
We say that a measure has \emph{full support} or is \emph{fully supported} if $\supp(\mu)=X$.
We say that $(X, T)$ is an \emph{E-system}
if it is transitive and admits a $T$-invariant Borel
probability measure with full support.

\subsection{Symbolic dynamics}
Below we have collected some basic facts from symbolic dynamics. The standard reference here is
the book of Lind and Marcus \cite{LM}.

Let $\{0,1\}^{\Z_+}$ be the space of infinite sequence of symbols in $\{0,1\}$
indexed by the non-negative integers.
Equip $\{0,1\}$ with the discrete topology and $\{0,1\}^{\Z_+}$ with the product topology.
The space $\{0,1\}^{\Z_+}$ is compact and metrizable. A compatible metric $\rho$ is given by
\[
\rho(x,y)=
\begin{cases}
  0,&x=y,\\
  2^{-J(x,y)}, &x\neq y,
\end{cases}
\]
where $J(x,y)=\min\{i\in{\Z_+}\colon x_i\neq y_i\}$.

A \emph{word} of length $n$ is a sequence $w=w_1w_2\dotsc w_n\in\{0,1\}^n$ and its \emph{length} is denoted by $|w|=n$.
The \emph{concatenation} of words $w=w_1w_2\dotsc w_n$ and $v=v_1v_2\dotsc v_m$
is the word $wv=w_1w_2\dotsc w_nv_1v_2\dotsc v_m$.
If $u$ is a word and $n\in\N$, then $u^n$ is the concatenation of $n$ copies of $u$
and $u^\infty$ is the sequence in $\{0,1\}^{\Z_+}$ obtained by infinite concatenation of the word $u$.
We say that a word $u=u_1u_2\dotsc u_k$ appears in $x=(x_i)\in\{0,1\}^{\Z_+}$ at position $t$
if $x_{t+j-1}=u_j$ for $j=1,2,\dotsc,k$.
For $x\in \{0,1\}^{\Z_+}$ and $i,j\in \Zp$, $i\leq j$
write $x_{[i,j]}=x_ix_{i+1}\ldots x_j$. Words $x_{[i,j)}$ and $x_{(i,j]}$,
$x_{(i,j)}$ are defined in the same way.

The \emph{shift map} $\sigma\colon\{0,1\}^{\Z_+}\to\{0,1\}^{\Z_+}$ is defined by
$\sigma(x)_n=x_{n+1}$ for $n\in{\Z_+}$.
It is clear that $\sigma$ is a continuous surjection.
The dynamical system $(\{0,1\}^{\Z_+},\sigma)$ is called the \emph{full shift}.
If $X$ is non-empty, closed  and $\sigma$-invariant (i.e. $\sigma(X)\subset X$),
then $(X, \sigma)$ is called a \emph{subshift}.

Given any collection $\F$ of words over $\{0,1\}$,
we define a \emph{subshift specified by $\F$},
denoted by $X_\F$,
as the set of all sequences from $\{0,1\}^{\Z_+}$ which do not contain any words from $\F$.
We say that $\F$ is a collection of \emph{forbidden words for $X_\F$}
as words from $\F$ are forbidden to occur in $X_\F$.

A \emph{cylinder} in $\{0,1\}^{\Z_+}$ is any set
$[u]=\{x\in X: x_0x_1\dotsc x_{n-1}=u\}$, where $u$ is a word of length $n$.
Note that the family of cylinders in $\{0,1\}^{\Z_+}$ is a base of the topology of $\{0,1\}^{\Z_+}$.
Let $X$ be a subshift of $\{0,1\}^{\Z_+}$.
The \emph{language of $X$}, denoted $\mathcal{L}(X)$, consists of all words that can appear in some $x\in X$,
i.e. $\mathcal{L}(X)=\set{x_{[i,j]}: x\in X, i\leq j}$.

For every word $u\in \mathcal{L}(X)$, let $[u]_X=X\cap [u]$.
Then $\{[u]_X\colon u\in\mathcal{L}(X)\}$ forms a base of the topology of $X$.
Let $\F=\{0,1\}^*\setminus\mathcal{L}(X)$, where $\{0,1\}^*$ is the collection of all finite words over $\{0,1\}$.
Then $X=X_\F$, that is, $\F$ is the set of forbidden words for $X$.

\begin{rem}
In some examples we will consider sequences indexed by positive integers $\N$ instead of $\Zp$.
That is, we identify
$\set{0,1}^\N$ with $\set{0,1}^\Zp$. It will simplify some calculations.
\end{rem}

\section{Multi-recurrent points }

\subsection{Definition and basic properties}

\begin{defn}\label{def:multirec}
Let $(X,T)$ be a dynamical system. A point $x\in X$ is called \emph{multi-recurrent}
if for every  $d\geq 1$ there exists a strictly increasing sequence $\{n_k\}_{k=1}^\infty$ in $\N$ such that
for each $i=1,2,\dotsc,d$ we have $T^{i n_k} x\to x$ as $k\to\infty$.
\end{defn}

In other words, a point $x\in X$ is multi-recurrent if and only if for every $d\ge 1$ the point $(x,\ldots,x)\in X^d$ is recurrent for $T\times T^2\times\dotsb\times T^d$.
Equivalently, $x$ is multi-recurrent if and only if for every $d\geq 1$
and every neighborhood $U$ of $x$
there exists $k\in\N$ such that $k,2k,\dotsc,dk\in N(x,U)$.

While we do not need such generality in the present paper,
observe that Definition~\ref{def:multirec} can be stated for $\Z^d$-actions in a similar manner.
A proof of the following observation is straightforward, thus we leave it to the reader.
\begin{lem}\label{lem:multi-rec-n}
Let $(X,T)$ be a dynamical system and $x\in X$.
Then the following conditions are equivalent:
\begin{enumerate}
  \item $x$ is a multi-recurrent point of $(X,T)$;
  \item $x$ is a multi-recurrent point of $(X,T^n)$ for some $n\in\N$;
  \item $x$ is a multi-recurrent point of $(X,T^n)$ for any $n\in\N$.
\end{enumerate}
\end{lem}

The following fact implies that every dynamical system contains a multi-recurrent point, because every dynamical system has a minimal subsystem.
Note that Lemma \ref{lem:G-delta} can also be deduced from properties of sets of multiple recurrence provided by \cite[Lemma~2.5]{Kra15}.
Results in \cite{Kra15} allow further analysis of return times of multi-recurrent points.

\begin{lem}\label{lem:G-delta}
Let $(X,T)$ be a dynamical system.
\begin{enumerate}
  \item\label{lem:G-delta:1} 
  The set of all multi-recurrent points of $(X,T)$ is a $G_\delta$ subset of $X$.
  \item\label{lem:G-delta:2} If $(X,T)$ is minimal, then the set of all multi-recurrent points is residual in $X$.
\end{enumerate}
\end{lem}
\begin{proof}
\eqref{lem:G-delta:1}: Given $d\geq 1$, let
\[R_d=\bigl\{y\in X\colon \exists n\geq 1\text{ such that }\rho(y,T^{in}y)<\tfrac{1}{d} \text{ for }i=0,1,\ldots,d\}.\]
It is clear that every $R_d$ is open, hence
 $R=\bigcap_{d=1}^\infty R_d$ is a $G_\delta$ subset of $X$.
It is easy to see that $R=\bigcap_{d=1}^\infty R_d$ is the set of all multi-recurrent points.

\eqref{lem:G-delta:2}: If $(X,T)$ is minimal, then it follows from the Topological Multiple Recurrence Theorem II
that $R_d$ is dense in $X$ for every $d\ge1$. Thus
$R=\bigcap_{d=1}^\infty R_d$ is residual in $X$.
\end{proof}

\begin{lem}\label{lem:urMultiRec}
If a dynamical system $(X,T)$ is uniformly rigid, then every point in $X$ is multi-recurrent.
\end{lem}
\begin{proof}
Fix $d\geq 1$. Since $(X,T)$ is uniformly rigid, for
every $\varepsilon>0$ there exists $n\in \N$ such that $\rho(T^nx,x)<\varepsilon/d$ for all $x\in X$.
Then
\[
\rho(x,T^nx)<\varepsilon/d,\; \rho(T^nx,T^{2n}x)<\varepsilon/d,\ldots,\rho(T^{(d-1)n}x,T^{dn}x)<\varepsilon/d,
\]
which shows that the diameter of $\{x,T^nx,T^{2n}x,\dotsc,T^{dn}x\}$ is less than $\varepsilon$.
It follows that $(x,\ldots,x)\in X^d$ is recurrent for $T\times T^2\times\dotsb\times T^d$.
But $d$ is arbitrary, hence $x$ is multi-recurrent.
\end{proof}

\begin{rem}\label{rem:distal}
It is shown in \cite[Proposition~9.16]{F81} that if a point is distal then it is multi-recurrent.
In particular, in a distal system every point is multi-recurrent.
\end{rem}

\begin{rem}
Notice that there exist minimal as well as non-minimal weakly mixing and uniformly rigid systems (see, respectively, \cite{GM89} and \cite{FHLO16}).
By Lemma~\ref{lem:urMultiRec}, every point in those systems is multi-recurrent.
None of these examples can be a subshift. Furthermore, a non-trivial strongly mixing dynamical system can never be uniformly rigid by~\cite{GM89}.
\end{rem}

One of the referees of this paper, motivated by the above remark,  suggested the following problem.

\begin{ques}
Is there a non-trivial weakly mixing subshift or any mixing dynamical system for which each point is multi-recurrent? Can such a system be minimal?
\end{ques}


In~\cite{W98} it is proved that if each pair in a dynamical $(X, T)$ is positively
recurrent under $T\times T$, then it has zero topological entropy (it is also a consequence of a result in \cite{BHR}).
Distal or uniformly rigid systems are examples of pointwise multi-recurrent systems which have zero topological entorpy.
But pointwise multi-recurrence does not imply zero topological entropy in general as shown below.

\begin{rem}
A dynamical system  $(X,T)$ is \emph{multi-minimal}
if for every $d\geq 1$ $(X^d,T\times  T^2\times\dotsb\times T^d)$ is minimal~\cite{M10}.
Clearly, every point in a multi-minimal system is multi-recurrent.
Note that by the proof of \cite[Proposition 3.5]{HYZ14}
there exists a multi-minimal system with positive topological entropy.
\end{rem}

The existence of a system constructed in the following theorem is probably a folklore, but we were unable to find it in the literature.

\begin{thm}\label{thm:minimial-not-multi-rec}
For every $d\geq 1$,  there is a minimal point $x$ in the full shift $(\{0,1\}^{\Z_+},\sigma)$
such that $(x,x,\ldots,x)\in X^{d}$ is recurrent under $\sigma \times \sigma^2\times \dotsb \times \sigma^{d}$
and $(x,x,\ldots,x)\in X^{d+1}$ is not recurrent under
$\sigma \times \sigma^2 \times \dotsb \times \sigma^{d}\times \sigma^{d+1}$.
\end{thm}
\begin{proof}
First we consider the case $d=1$ and then the general case.
For $d=1$, we define the local rule of a substitution by
\begin{align*}
\tau\colon& 1\longrightarrow 1101,\\
&0\longrightarrow 0101,
\end{align*}
and then extend it to all finite words over $\{0,1\}$ putting inductively $\tau(uv)=\tau(u)\tau(v)$.
Let $x=(x_i)_{i=0}^\infty=\lim_{k\to\infty} \tau^k(1) 0^\infty$ be a fixed point of $\tau$.
It is easy to check that $x\in\{0,1\}^{\Z_+}$ is a minimal point.

We claim that $x_i=1$ if and only if $i=0$ or $i=4^{m}(2n+1)$ for some $n,m\in\Z_+$.
It will follow that $x_i=0$ if and only if $i=2\cdot 4^{m}(2 n+1)$ for some $n,m\in\Z_+$.

These conditions are clearly
true for $i=0,1,2,3$. Now fix any $i\geq 0$ and assume that our claim holds for $i$.
We will show that the claim also holds for
$4i,4i+1,4i+2,4i+3$. We have two cases to consider.

If $x_i=1$, then by the claim $i=4^m(2n+1)$ for some $m,n\in\Z_+$. By the definition of substitution
$x_{[4i,4i+3]}=\tau(x_i)=\tau(1)$, so
\begin{itemize}
  \item $x_{4i}=1$ and $4i=4^{m+1}(2n+1)$;
  \item $x_{4i+1}=1$ and $4i+1=4^{m+1}(2n+1)+1=2(2\cdot 4^m(2n+1))+1$;
  \item $x_{4i+2}=0$ and $4i+2=4^{m+1}(2n+1)+2=2(2\cdot 4^m(2n+1)+1)$;
  \item $x_{4i+3}=1$ and $4i+3=4^{m+1}(2n+1)+3=2(2\cdot 4^m(2n+1)+1)+1$.
\end{itemize}

If $x_i=0$, then $i=2\cdot4^m\cdot n$ for some $m,n\in\Z_+$. Then
$x_{[4i,4i+3]}=\tau(0)$ and we have:
\begin{itemize}
  \item $x_{4i}=0$ and $4i=2\cdot 4^{m+1}\cdot n$;
  \item $x_{4i+1}=1$ and $4i+1=2\cdot 4^{m+1}\cdot n+1=2(4^{m+1}\cdot n)+1$;
  \item $x_{4i+2}=0$ and $4i+2=2\cdot 4^{m+1}\cdot n+2=2(4^{m+1}\cdot n +1)$;
  \item $x_{4i+3}=1$ and $4i+3=2\cdot 4^{m+1}\cdot n+3=2(4^{m+1}\cdot n+1)+1$.
\end{itemize}
This ends the proof of the claim.

The point $x$ is minimal, hence
it is recurrent under $\sigma$.
By the claim, it is clear that if $i\in\N$ and $x_i=1$ then $x_{2i}=0$.
So $(x,x)$ is not recurrent under $\sigma\times\sigma^2$, because it will never return to $[1]\times [1]$.

For the case $d\geq 2$, we extend the above idea.
We define a local rule of a substitution by
\begin{align*}
\tau\colon& 1\longrightarrow 1a_1\ldots a_{(d+1)^2-1},\\
&0\longrightarrow 0a_1\ldots a_{(d+1)^2-1},
\end{align*}
where $a_{j}=0$ for $j \equiv 0 \bmod(d+1)$ and $a_j=1$ otherwise.
Let $x=\lim_{k\to\infty} \tau^k(1) 0^\infty$ be a fixed point of $\tau$. As above, $x$ is a minimal point.

For every $k\in\N$, $x$ can be expressed as $x=[\tau^k(1)]^{d+1}\tau^k(0) \dotsc$,
so $(x,x,\dotsc,x)\in X^d$ is recurrent under $\sigma\times \sigma^2\times \dotsb\times \sigma^d$.
Analogously to the case $d=1$, we prove that if $j\in\N$ and $x_j=1$ then $x_{(d+1)j}=0$. The details are left to the reader.
So $(x,x,\dotsc,x)\in X^{d+1}$ is not recurrent under $\sigma\times \sigma^2\times \dotsb\times \sigma^d\times \sigma^{d+1}$.
\end{proof}

\subsection{Multi-recurrent points and factor maps.}
Let $\pi\colon (X,T)\to (Y,S)$ be a factor map.
It is well known that if $y\in Y$ is a recurrent point of $S$,
then there is a recurrent point $x\in X$ of $T$ with $\pi(x)=y$.
In this subsection we investigate if this result holds for multi-recurrent points.
It turns out that it is still the case for distal extensions but may fail
for proximal extensions.

\begin{prop} \label{prop:multi-rec-lift}
Let $\pi\colon (X,T)\to (Y,S)$ be a factor map.
\begin{enumerate}
  \item\label{prop:multi-rec-lift:1}  If $x\in X$ is multi-recurrent, then so is $\pi(x)$.
  \item\label{prop:multi-rec-lift:2} If $y\in Y$ is multi-recurrent and $\pi^{-1}(y)$ consists of a single point $x$,
then $x$ is also multi-recurrent.
\end{enumerate}
\end{prop}
\begin{proof}
\eqref{prop:multi-rec-lift:1}: It is a direct consequence of continuity of $\pi$.

\eqref{prop:multi-rec-lift:2}: Since $\pi^{-1}(y)=\{x\}$,
for every neighborhood $U$ of $x$ there exists a neighborhood $V$ of $y$ such that
$\pi^{-1}(V)\subset U$.
Therefore $N(y,V)\subset N(x,U)$. It follows that if $y$ is multi-recurrent, then so is $x$.
\end{proof}

By Remark~\ref{rem:distal} every distal system is multi-recurrent.
In particular, every equicontinuous system is multi-recurrent. Therefore the projection
of minimal dynamical system onto its maximal equicontinuous factor maps every point onto a multi-recurrent point.
It turns out that the system presented in Theorem~\ref{thm:minimial-not-multi-rec} is a proximal extension of its
maximal equicontinuous factor and there is a fiber not containing any multi-recurrent points.

\begin{prop}\label{prop:multi-rec-not-lift}
There exist two dynamical systems $(X,T)$ and $(Y,S)$, a proximal factor map $\pi \colon (X,T)\to (Y,S)$
and a point $y\in Y$ which is multi-recurrent but $\pi^{-1}(y)$ does not contain any multi-recurrent points.
\end{prop}

\begin{proof}
Let $\tau$ be a local rule of a substitution defined by
\begin{align*}
\tau\colon& 1\longrightarrow 1101,\\
&0\longrightarrow 0101,
\end{align*}
i.e. $\tau$ is the substitution from the proof of Theorem~\ref{thm:minimial-not-multi-rec}.
Let $x=\lim_{n\to\infty}\tau^n(1)0^\infty$ and $z=\lim_{n\to\infty}\tau^n(0)0^\infty$ be fixed points of $\tau$.
Let $X=\overline{\Orb(x,\sigma)}$.
Then $X$ is a minimal set  and $z\in X$.

Observe that $z_0=0$, $z_k=1$ for $k=4^{m}(2n+1)$ and $z_k=0$ for $k=2\cdot 4^{m}(2n+1)$.
In particular one has  $z_i=x_i$ for $i>0$ (see the proof of Theorem~\ref{thm:minimial-not-multi-rec}).
Note that if $z_j=0$ for some $j>0$ then $z_{2j}=1$ and if $z_j=1$ then $z_{2j}=0$.
Neither $(x,x)$ nor $(z,z)$  is recurrent under $\sigma\times \sigma^2$.

Denote $k_n=|\tau^n(1)|=4^n$ and observe that position of $11$ uniquely identifies position
of $\tau(1)$ in $x=\tau(x)$. By the same argument $\tau(1)\tau(1)$ identifies uniquely beginning of $\tau^2(1)$
in $x$, etc. In other words, blocks $\tau^n(0)$ and $\tau^n(1)$ form a code for every $n\geq 1$ and
hence there is a unique
decomposition of $x$ into blocks from $\set{\tau^n(0),\tau^n(1)}$.
But $X$ is the closure of the orbit of $x$ which yields
that for any $v\in X$ and any $n\geq 1$ there is a uniquely determined
infinite concatenation $\{w^{(n)}\}_{j=1}^\infty$ of blocks over $\set{\tau^n(0),\tau^n(1)}$
and a block $u_n$ of length $0\leq |u_n|< k_n$ such that $v=u_nw^{(n)}_1w^{(n)}_2w^{(n)}_3\dotsb$.

With every $n$ associate a natural projection $\xi_n\colon \Z_{k_{n+1}}\to \Z_{k_n}$,
$\xi_n(x)=x (\bmod k_n)$. Then we obtain a well defined inverse limit
$$
Y=\underleftarrow{\lim}(\Z_{k_n},\xi_n)=\set{(j_1,j_2,\ldots) : \xi_n(j_{n+1})=j_n}\subset \prod\Z_{k_n}
$$
Addition in $Y$ is coordinatewise, modulo $k_n$ on each coordinate $n$. Endowed with the product topology over the discrete topologies in $Z_{k_n}$ space $Y$ becomes a topological group satisfying the four properties characterizing odometers (see \cite{Dow}). Let $S\colon Y\to Y$
be defined by $S(j_1,j_2,\ldots)=(j_1+1,j_2+1,\ldots)$. Then $Y=\overline{\Orb((0,0,\ldots),S)}$ and $(Y,S)$ is a minimal dynamical system (an odometer).

With every $v\in X$ we can associate a sequence  $j^{(v)}=(j_1^{(v)},j_2^{(v)},\ldots)\in Y$ given by $j^{(v)}_n=k_n-|u_n|\; (\bmod  k_n)$.
This way we obtain a natural factor map $\pi \colon (X,\sigma)\to (Y,S)$,
$v \mapsto j^{(v)}$. Note that if $\pi(u)=\pi(u')$ then
for every $n\geq 1$ taking $k=|\tau^{n+1}(0)|-j_{n+1}$
provides a decomposition $\sigma^k(u),\sigma^k(u')\in \set{\tau^n(0),\tau^n(1)}^{\Zp}$
which in turn implies that $u,u'$ share arbitrarily long common word of symbols (e.g. $\tau^n(0)$),
and as a consequence $u,u'$ form a proximal pair.
This proves that $\pi$ is a proximal extension.
Denote $y=\pi(x)$.

To finish the proof observe that if $u\in X$ and $\pi(u)_n=0$ then $u\in \set{\tau^n_2(0),\tau^n_2(1)}^{\Zp}$ by the definition of $\pi$.
But if $\tau^n(0)$ is a prefix of $u$ (the same for $\tau^n(1)$ and $\pi(u)_{n+1}=0$
then $\tau^{n+1}(0)$ must be a prefix of $u$ (resp.  $\tau^{n+1}(1)$ is a prefix).
Therefore, if we put $y=(0,0,0,\ldots)$ then $\pi^{-1}(y)=\set{x,z}$ and every point in $(Y,S)$ is multi-recurrent (it is a distal system and so Remark~\ref{rem:distal} applies). 
\end{proof}

To prove that multi-recurrent points can be lifted by distal extensions,
we apply the theory of enveloping semigroup.
Let $(X,T)$ be a dynamical system.
Endow $X^X$ with the product topology.
By the Tychonoff theorem, $X^X$ is a compact Hausdorff space.
The \emph{enveloping semigroup} of $(X,T)$, denoted by $E(X,T)$, is defined as the
closure of the set $\{T^n: n\in\Z_+\}$ in $X^X$.
We refer the reader to the book~\cite{Aus}  for more details (see also \cite{AAG}).

\begin{thm}
Let $\pi \colon (X,T)\to (Y,S)$ be a factor map, let $d\geq 1$ and
assume that $y\in Y$ is recurrent under $S  \times S^2\times \dotsb \times S^d$ .
If $x\in \pi^{-1}(y)$ is such that the pair $(x,z)$ is distal for any $z\in \pi^{-1}(y)$ with $z\neq x$,
then $x$ is recurrent under $T\times T^2\times \dotsb \times T^d$.
In particular, if $y$ is multi-recurrent then so is $x$.
\end{thm}
\begin{proof}
Let $\pi_d=\pi\times \pi\times \dotsb \times \pi\colon (X^d, T\times T^2\times \dotsb \times T^d)\rightarrow
(Y^d, S \times  S^2\times \dotsb \times S^d)$. Then $\pi_d$ is a factor map.
There exists a unique onto homomorphism
$\theta \colon E(X^d, T\times T^2\times \dotsb \times T^d) \to E(Y^d, S  \times S^2\times \dotsb \times S^d)$ such
that $\pi_d(pz)=\theta(p)\pi_d(z)$ for any $p\in E(X^d, T\times \ldots \times T^d)$ and $z\in X^d$ (see Theorem~3.7 in \cite{Aus}).
Since $(y,\ldots, y)$ is recurrent under the action of $S \times S^2\times \dotsb \times S^d$,
by \cite[Proposition~2.4]{AAG} there is an idempotent $u\in E(Y^d, S  \times S^2\times \dotsb \times S^d)$
such that $u(y,\ldots,y)=(y,\ldots,y)$. If we denote $J=\theta^{-1}(u)$
then clearly it is a closed subsemigroup of $E(X^d, T\times T^2\times \dotsb \times T^d)$
and so by Ellis-Numakura Lemma there is an idempotent $v\in J$.

Observe that
\[
\pi_d(v(x,\ldots,x))=\theta(v)\pi_d(x,\ldots,x)=u(y,\ldots,y)=(y,\ldots,y),
\]
hence each coordinate of $v(x,\ldots,x)$ belongs to $\pi^{-1}(y)$.
Furthermore, since $v$ is an idempotent,
we have $v(v(x,\ldots,x))=v(x,\ldots,x)$,
thus again by \cite[Proposition~2.4]{AAG} we obtain that $v(x,\ldots,x)$ and $(x,\ldots,x)$
are proximal under $T\times T^2\times \dotsb \times T^d$,
and therefore each coordinate of $v(x,\ldots,x)$ is proximal with $x$ (under the action of $T$).
But the pair $(x,z)$ is distal for any $z\in \pi^{-1}(y)$ with $z\neq x$,
which immediately implies that $v(x,\ldots,x)=(x,\ldots,x)$.
Since $v$ is an idempotent,
it is equivalent to say that $(x,\ldots,x)$ is recurrent under
$T\times T^2\times \dotsb \times T^d$ which ends the proof.
\end{proof}

\begin{cor}
Let $\pi\colon (X,T)\to (Y,S)$ be a factor map.
If $\pi$ is distal, then a point $x\in X$ is multi-recurrent if and only if so is $\pi(x)$.
\end{cor}

\subsection{The measure of multi-recurrent points}
It follows from the Poincar\'e recurrence theorem that
almost every point is recurrent for any invariant measure
(see \cite[Theorem 3.3]{F81}). A similar connection holds between
multi-recurrent points and multiple recurrence in ergodic theory.

\begin{thm}\label{thm:almostallmr}
Let $(X,T)$ be a dynamical system and  $\mu$ be a $T$-invariant Borel probability measure on $X$.
Then $\mu$-almost every point of $X$ is multi-recurrent for $T$.
\end{thm}
\begin{proof}
Choose a countable base $\{B_i\}_{i=1}^\infty$ for topology of $X$.
For every $i\in\N$, let
\[A_i=\bigcup_{d=1}^\infty
\Bigl(B_i \setminus \bigcup_{n=1}^\infty B_i\cap T^{-n}B_i\cap T^{-2n}B_i\cap \dotsb\cap T^{-dn}B_i\Bigr).\]
Note that a point $x$ is not multi-recurrent if and only if
there exist $d\geq 1$ and $i\in\N$ such that $x\in B_i$ but
 $x\not\in B_i\cap T^{-n}B_i\cap \dotsb\cap T^{-dn}B_i$ for all $n\in\N$.
Therefore $\bigcup_{i=1}^\infty A_i$ is the collection of non-multi-recurrent points of $(X,T)$.
By the Multiple Recurrence Theorem, $\mu(A_i)=0$ for every $i\geq 1$.
Then $\mu(\bigcup_{i=1}^\infty A_i)=0$.
\end{proof}

\begin{cor}
If a dynamical system $(X,T)$ admits  an ergodic invariant Borel probability measure $\mu$ with full support,
then there exists a dense $G_\delta$ subset $X_0$ of $X$ with full $\mu$-measure such that
every point in $X_0$ is both transitive and multi-recurrent.
\end{cor}

\begin{proof}
Since $\mu$ is ergodic, then the set of all transitive points is a dense $G_\delta$ subset of $X$ and has full $\mu$-measure.
By Lemma~\ref{lem:G-delta} and Theorem \ref{thm:almostallmr},
the set of all multi-recurrent point is also a dense $G_\delta$ subsets of $X$ and has full $\mu$-measure.
Then the intersection of those two sets is as required.
\end{proof}

Using results on multiple recurrence developed by Furstenberg in \cite{F77},
we strengthen Theorem~\ref{thm:almostallmr} as follows.

\begin{thm}\label{thm:measure-density-multi-rec}
Let $(X,T)$ be a dynamical system.
For every $T$-invariant Borel probability measure $\mu$ on $X$,
there exists a Borel subset $X_0$ of $X$ with $\mu(X_0)=1$ such that
for every $x\in X_0$, every $d\in\N$ and every neighborhood $U$ of $x$
the set $N_{T\times T^2\times\ldots \times T^d}((x,\ldots,x),U\times\dotsb \times U)$ has positive upper density.
\end{thm}

\begin{proof}
For every $d\in\N$ and every $\delta>0$, let $A_{d,\delta}$ be the collection of all points $x\in X$
for which there exists a neighborhood $U$ of $x$ with $\diam(U)<\delta$
such that the set
\[N_{T\times T^2\times\ldots \times T^d}((x,\ldots,x),U\times \dotsb \times U)\]
has positive upper density.

Let $\mu$ be an ergodic $T$-invariant Borel probability measure on $X$.
We are going to show that $\mu(A_{d,\delta})=1$ for every $d\in\N$ and every $\delta>0$.
First we show that $A_{d,\delta}$ is Borel measurable.
To this end, for every $t>0$ and every $n,m\in\N$,
let $A_{d,\delta}(t,n,m)$ be the collection of all points $x\in X$ such that
there exists an neighborhood $U$ of $x$ with $\diam(U)<\delta$
satisfying
\[\frac{1}{n}\#\bigl(N_{T\times T^2\times\ldots \times T^d}((x,\ldots,x),U\times \dotsb \times U)\cap [0,n-1]\bigr)>t-\tfrac{1}{m}.\]
It is clear that $A_\delta(t,n,m)$ is an open subset of $X$ and
\[A_{d,\delta}=\bigcup_{k=1}^\infty \bigcap_{m=1}^\infty \bigcup_{n=m}^\infty A_{d,\delta}(\tfrac{1}{k},n,m).\]
It follows that $A_{d,\delta}$ is Borel measurable.

If $\mu(A_{d,\delta})<1$, then we can choose a Borel subset $B\subset X\setminus A_\delta$ with $\diam(B)<\delta/3$ and $\mu(B)>0$.
For any $x\in X$, let
\[g(x)=\limsup_{N\to\infty}\frac{1}{N}\sum_{i=0}^{N-1}\mathbf{1}_{B\cap T^{-i}B\cap \dotsb\cap T^{-id}B}(x).\]
Then $g$ is also Borel measurable and $0\leq g(x)\leq 1$ for any $x\in X$.
By the Fatou lemma and \cite[Theorem 7.14]{F81}, we have
\begin{align*}
\int_X g(x) d\mu(x)&\geq
\limsup_{N\to\infty}\frac{1}{N}\int_X\sum_{i=0}^{N-1}\mathbf{1}_{B\cap T^{-i}B\cap \dotsb\cap T^{-id}B}(x)d\mu(x)\\
&\geq\liminf_{N\to \infty}\frac{1}{N}\sum_{i=0}^{N-1} \mu(B\cap T^{-i}B\cap \dotsb\cap T^{-id}B)>0.
\end{align*}
Clearly $g(x)=0$ for any $x\not\in B$, hence
there exists some $x\in B$ such that $g(x)>0$.
Let $U=B(x,\tfrac{2}{3}\delta)$. Then $B\subset U$ and
the upper density of $N_{T\times T^2\times\dotsb \times T^d}((x,\ldots,x),U\times \dotsb \times U)$
is not less than $g(x)$. We obtain that  $x\in A_{d,\delta}$, which is a contradiction.

Therefore $\mu(A_{d,\delta})=1$ for every ergodic measure $\mu$, every $d\in\N$ and every $\delta>0$.
Let
\[X_0=\bigcap_{d=1}^\infty\bigcap_{k=1}^\infty A_{d,\frac{1}{k}}.\]
Then $\mu(X_0)=1$ for every ergodic measure, and by the ergodic decomposition the same holds for any $T$-invariant measure. Therefore $X_0$ is as required.
\end{proof}

\begin{rem}
Assume that pointwise convergence of multiple averages holds for $\mu$, that is,
for every $d\in\N$ and $f_1,f_2,\dotsc,f_d\in L^\infty(\mu)$,
\[\frac{1}{N}\sum_{n=0}^{N-1} f_1(T^nx) f_2(T^{2n}x)\dotsb f_d(T^{dn}x) \text{ converges } \mu\ a.e..\]
Then the proof of Theorem~\ref{thm:measure-density-multi-rec} can be modified by replacing $\limsup$ in the definition of $g$ by $\liminf$, and the modified proof yields that
for every $x\in X_0$, every $d\in\N$ and every neighborhood $U$ of $x$
the set $N_{T\times T^2\times\ldots \times T^d}((x,\ldots,x),U\times\dotsb \times U)$ has positive lower density.
Unfortunately, the pointwise convergence of multiple averages for general ergodic measures is still an open problem.
It was proved recently that the pointwise convergence of multiple averages holds for distal measures (see \cite{HSY2014}).
\end{rem}

Glasner proved in \cite{G94} that if a minimal system $(X,T)$ is topologically weakly mixing,
then there is a dense $G_\delta$ subset $X_0$ such that for each $x\in X_0$,
the orbit of $(x,\ldots,x)$ is dense in $X^d$ under
$T\times T^2\times\ldots \times T^d$. Below we present an analogous result
for systems possessing a fully weakly mixing invariant measure.
Note that Lehrer \cite{L87} proved a variant of the Jewett-Krieger theorem,
which implies that there are topologically weakly mixing minimal systems
without weakly mixing invariant measures. Therefore our result complements Glasner's theorem.

\begin{thm}\label{thm:weak-mixing-muli-rec}
Let $(X,T)$ be a dynamical system.
If there exists a weakly mixing, fully supported $T$-invariant Borel probability
measure $\mu$ on $X$, then
there exists a Borel subset $X_0$ of $X$ with $\mu(X_0)=1$ such that
for every $x\in X_0$, every $d\in\N$, and every non-empty open subsets $U_1,U_2,\dotsc,U_d$ of $X$
the set
\[
N_{T\times T^2\times\ldots \times T^d}\big((x,x,\ldots,x),U_1\times U_2\times\dotsb \times U_d\big)
\]
has positive upper density.
\end{thm}

\begin{proof}
For every $d\in\N$ and every $\delta>0$, let $A_{d,\delta}$ be the collection of all points $x\in X$
such that there exists an open cover $\{U_i\}_{i=1}^\ell$ of $X$ with $\diam(U_i)<\delta$ for $i=1,\dotsc,\ell$ and
such that for every $\alpha\in \{1,2,\dotsc,\ell\}^d$ the set
$N_{T\times T^2\times\ldots \times T^d}((x,x,\ldots,x),U_{\alpha(1)}\times U_{\alpha(2)}
\times\dotsb \times U_{\alpha(d)})$
has positive upper density.

Following the same lines as in the proof of Theorem~\ref{thm:measure-density-multi-rec}
we obtain that $A_{d,\delta}$ is Borel measurable.
We are going to show that $\mu(A_{d,\delta})=1$.

If $\mu(A_{d,\delta})<1$, there exists a Borel set $W_0\subset X\setminus A_\delta$
with $\diam (W_0)<\delta/2$ and $\mu(W_0)>0$.
Fix an open cover $\{U_i\}_{i=1}^p$ of $X$ with $\diam(U_i)<\delta$ for $i=1,\dotsc,\ell$.
Enumerate  $\{1,2,\dotsc,p\}^d$ as $\{\alpha_1,\alpha_2,\dotsc,\alpha_k\}$ with $k=p^d$.

First note that $\mu(U_j)>0$ for $i=1,2,\dotsc,\ell$ since $\mu$ has the full support.
For every $x\in X$, let
\[g_1(x)=\limsup_{N\to\infty}\frac{1}{N}\sum_{l=0}^{N-1}
\mathbf{1}_{{W_0}\cap T^{-l}U_{\alpha_1(1)}\cap \ldots\cap T^{-ld}U_{\alpha_1(d)}}(x).\]
Then $g_1$ is also Borel measurable and $0\leq g_1(x)\leq 1$ for any $x\in X$.
The measure $\mu$ is weakly mixing, hence we can apply \cite[Theorem~2.2]{F77} obtaining that
\[
\lim_{N\to \infty}\frac{1}{N}\sum_{l=1}^{N}\mathbf{1}_{{W_0}\cap T^{-l}U_{\alpha_1(1)}\cap \ldots\cap T^{-ld}U_{\alpha_1(d)}}(x)=\mathbf{1}_{W_0}(x) \prod_{l=1}^d\mu(U_{\alpha_1(l)})
\]
in $L^2(X)$. In particular $\int_X g_1(x)d\mu>0$.
Clearly $g_1(x) = 0$ for any $x\not\in W_0$.
Then there exists a Borel set $W_1\subset W_0$ with $\mu(W_1)>0$ and $g_1(x)>0$ for any $x\in W_1$.
Note that for every $x\in W_1$ the upper density of
$N_{T\times T^2\times\ldots \times T^d}((x,x,\ldots,x),U_{\alpha_1(1)}\times U_{\alpha_1(2)}
\times\dotsb \times U_{\alpha_1(d)})$ is not less than $g_1(x)$.

Working by induction, for every $i=1,2,\dotsc,k$, we can construct a Borel set $W_i\subset W_{i-1}$
with $\mu(W_i)>0$ such that for every $x\in W_i$ the set
$N_{T\times T^2\times\ldots \times T^d}((x,x,\ldots,x),U_{\alpha_i(1)}\times U_{\alpha_i(2)}
\times\dotsb \times U_{\alpha_i(d)})$ has positive upper density.
This implies that for every $x\in W_k$ and every $\alpha\in \{1,2,\dotsc,\ell\}^d$ the set
$N_{T\times T^2\times\ldots \times T^d}((x,x,\ldots,x),U_{\alpha(1)}\times U_{\alpha(2)}
\times\dotsb \times U_{\alpha(d)})$
has positive upper density.
Then $W_k\subset A_{d,\delta}$, which is a contradiction, hence $\mu(A_{d,\delta})=1$.

To finish the proof, it is enough to put
\[X_0=\bigcap_{d=1}^\infty\bigcap_{k=1}^\infty A_{d,\frac{1}{k}}.\]
since $\mu(X_0)=1$ and $X_0$ is as required.
\end{proof}

\begin{rem}
One can modify the proof of Theorem~\ref{thm:weak-mixing-muli-rec},
by replacing $A_{d,\delta}$ by $A_{d,\delta}'$ defined as the collection of all points $x\in X$
such that there exists an open cover $\{U_i\}_{i=1}^\ell$ of $X$ with $\diam(U_i)<\delta$ for $i=1,\dotsc,\ell$
for which the set
\[
N_{T\times T^2\times\ldots \times T^d}((x,x,\ldots,x),U_{\alpha(1)}\times U_{\alpha(2)}
\times\dotsb \times U_{\alpha(d)})\]
is not empty for every $\alpha\in \{1,2,\dotsc,\ell\}^d$. Then one obtains that $A_{d,\delta}'$ is a dense open subset of $X$ and
\[X_0'=\bigcap_{d=1}^\infty\bigcap_{k=1}^\infty A_{d,\frac{1}{k}}'\]
is a dense $G_\delta$ subset of $X$ with full $\mu$-measure.
Moreover, for every $d\in\N$ and every $x\in X_0'$, the orbit of $(x,x,\ldots,x)$ is dense in $X^d$ under
$T\times T^2\times\ldots \times T^d$.
Since $(X^d, T\times T^2\times\ldots \times T^d)$ is an E-system,
by \cite[Lemma 3.6]{HPY07} we know that
for every $x\in X_0'$, every $d\in\N$ and every non-empty open subsets $U_1,U_2,\dotsc,U_d$ of $X$
the set $N_{T\times T^2\times\ldots \times T^d}((x,x,\ldots,x),\allowbreak U_1\times U_2\times\dotsb \times U_d)$
has positive upper Banach density, but we cannot conclude that it has positive upper density. On the other hand, we do not know whether the set $X_0$ constructed in Theorem~\ref{thm:weak-mixing-muli-rec} is residual.
\end{rem}

\section{Van der Waerden systems and \texorpdfstring{$\AP$}{AP}-recurrent points}
In this section we introduce the concept of a van der Waerden system. We explore how this notion relates
to the behaviour of multi-recurrent points and $\AP$-recurrent points.

\begin{defn}
We say that a dynamical system $(X,T)$  is a \emph{van der Waerden system}
if it satisfies the topological multiple recurrence property, that is for every non-empty open set $U\subset X$ and
every $d\in\N$ there exists an $n\in\N$ such that
\[U\cap T^{-n}U\cap T^{-2n}U\cap \dotsb \cap T^{-dn}U\neq\emptyset.\]
\end{defn}

By the Topological Multiple Recurrence Theorem, we know that every minimal system is a van der  Waerden system.
It follows from the ergodic Multiple Recurrence Theorem that every E-system is a van der  Waerden system.

It is easy to see that if $(X,T)$ is  a van der Waerden system, then the relation
$R=\bigcap_{d=1}^\infty R_d$ is residual, where
\[R_d=\bigl\{y\in X\colon \exists\, n\geq 1\text{ such that }\rho(y,T^{in}y)<\tfrac{1}{d} \text{ for }i=0,1,\ldots,d\}.\]
As a corollary, we obtain the following (cf. Lemma~\ref{lem:G-delta}).

\begin{lem}\label{lem:vdW-AP}
A dynamical system $(X,T)$ is  a van der Waerden system
if and only if it has a dense set of multi-recurrent points.
\end{lem}

By Lemmas ~\ref{lem:vdW-AP} and \ref{lem:multi-rec-n}, we have the following result.
\begin{prop}\label{prop:vdW-n}
Let $(X,T)$ be a dynamical system.
Then the following conditions are equivalent:
\begin{enumerate}
\item $(X,T)$ is a van der Waerden system;
\item $(X,T^n)$ is a van der Waerden system for some $n\in\N$;
\item $(X,T^n)$ is a van der Waerden system for any $n\in\N$.
\end{enumerate}
\end{prop}

Lemma~\ref{lem:urMultiRec} implies that every point
in a uniformly rigid system is multi-recurrent.
Then by Lemma~\ref{lem:vdW-AP} every uniformly rigid system is a van der Waerden system.
By~\cite{AAB,GW93}, every almost equicontinuous system is uniformly rigid.
We have just proved the following.
\begin{prop}\label{rem:aeq_vdW}
Every almost equicontinuous system is also a van der Waerden system.
\end{prop}

Moothathu introduced $\Delta$-transitive systems in \cite{M10}. Recall that a dynamical system $(X,T)$ is \emph{$\Delta$-transitive}
if for every $d\in\N$ there exists $x\in X$ such that the diagonal $d$-tuple
$(x,x,\dotsc,x)$ has a dense orbit under the action of $T\times T^2\times \dotsb \times T^d$.

\begin{prop} \label{prop:Delta-transitive}
If a dynamical system $(X,T)$ is $\Delta$-transitive, then it is a van der Waerden system.
\end{prop}
\begin{proof}
Let $U$ be a non-empty open subset of $X$ and fix any
$d\in\N$. There exists $x\in X$ such that  diagonal $d$-tuple
$(x,x,\dotsc,x)$ has a dense orbit under the action of $T\times T^2\times \dotsb \times T^d$.
Then there exists $n\in\N$ such that $T^nx\in U, T^{2n}x\in U,\dotsc, T^{dn}x\in U$ and thus
\[T^n x\in U\cap T^{-n}U\cap \dotsb\cap T^{-(d-1)n}U.\]
This shows that $(X,T)$ is a van der Waerden system.
\end{proof}

By Proposition~\ref{prop:multi-rec-not-lift}, multi-recurrent points may not be lifted through factor maps.
To remove this disadvantage,  we introduce the following slightly weaker notion of \emph{$\AP$-recurrent} point.
As we will see later, it is possible to characterize van der Waerden systems through $\AP$-recurrent points.

\begin{defn}
A point $x\in X$ is  \emph{$\AP$-recurrent} if $N(x,U)$ is an AP-set for every open neighborhood $U$ of $x$ .
\end{defn}

\begin{rem}
It is clear that every multi-recurrent point is $\AP$-recurrent
and every $\AP$-recurrent point is recurrent. The notion of $\AP$-recurrent points can be seen as an intermediate
notion of recurrence.
By Proposition~\ref{prop:vdW-sys-transitive}, every minimal point is $\AP$-recurrent since minimal systems
are van der Wearden systems.
But by Theorem \ref{thm:minimial-not-multi-rec} there exist some minimal points which are not multi-recurrent.
Those minimal points are $\AP$-recurrent but not multi-recurrent.
Every transitive point of the dynamical system presented in the proof of Proposition \ref{prop:mixing-not-vdW}
is not $\AP$-recurrent. So those transitive points are recurrent but not $\AP$-recurrent.
\end{rem}

\begin{lem}\label{lem:G-delta-AP}
Let $(X,T)$ be a dynamical system.
\begin{enumerate}
\item\label{lem:G-delta-AP:1} The collection of all $\AP$-recurrent points of $(X,T)$ is a $G_\delta$ subset of $X$.
\item\label{lem:G-delta-AP:2} $(X,T)$ is a van der Waerden system
if and only if it has a dense set of $\AP$-recurrent points.
\end{enumerate}
\end{lem}
\begin{proof}
\eqref{lem:G-delta-AP:1}: Given $d\geq 1$, let
\[Q_d=\bigl\{y\in X\colon \exists\, n,a\geq 1\text{ such that }\rho(y,T^{in+a}y)<\tfrac{1}{d} \text{ for }i=0,1,\ldots,d\}.\]
It is clear that every $Q_d$ is open, hence
$Q=\bigcap_{d=1}^\infty Q_d$ is a $G_\delta$ subset of $X$.
It is easy to see that $Q=\bigcap_{d=1}^\infty Q_d$ is the set of all $\AP$-recurrent points.

\eqref{lem:G-delta-AP:2}: First note that by Lemma~\ref{lem:vdW-AP} every van der Waerden system has dense set of multi-recurrent points, hence $\AP$-recurrent points are dense.

On the other hand, if $x$ is $\AP$-recurrent and $x\in U$ then for every $d\geq 1$
there are $a,n\geq 1$ such that $T^{a+in}x\in U$ for every $i=0,1,\ldots d$ and so
\[T^a x \in U\cap T^{-n}U\cap T^{-2n}U\cap \dotsb \cap T^{-dn}U\]
completing the proof.
\end{proof}

We have the following connection between $\AP$-recurrent points and their orbit closures.
\begin{prop} \label{prop:AP-rec-vdW}
Let $(X,T)$ be a dynamical system and $x\in X$.
Then $x$ is $\AP$-recurrent if and only if $(\overline{\Orb(x,T)},T)$ is a van der Waerden system.
\end{prop}
\begin{proof}
If $x$ is $\AP$-recurrent, then every point in the orbit of $x$ is also $\AP$-recurrent.
By Lemma~\ref{lem:G-delta-AP}, $(\overline{\Orb(x,T)},T)$ is a van der Waerden system.

Now assume that $(\overline{\Orb(x,T)},T)$ is a van der Waerden system.
By Lemma~\ref{lem:G-delta-AP},
$(\overline{\Orb(x,T)},T)$ has a dense set of $\AP$-recurrent points.
Fix an open neighborhood $U$ of $x$. It suffices to show that $N(x,U)\in \AP$.
Choose an $\AP$-recurrent points $y$ in $U$.
For every $d\geq 1$, there exist $k,n\in\N$ such that
\[T^ky\in U, T^{k+n}y\in U, T^{k+2n}y\in U,\dotsc, T^{k+dn}y\in U.\]
By continuity of $T$, there exists an open neighborhood $V$ of $y$ such that for any $z\in V$ we have
\[T^kz\in U, T^{k+n}z\in U, T^{k+2n}z\in U,\dotsc, T^{k+dn}z\in U.\]
Since $y\in\overline{\Orb(x,T)}$, there exists $m\geq 0$ such that $T^{m}x\in V$.
Then
 \[T^{m+k} x\in U, T^{m+k+n}x\in U, T^{m+k+2n}x\in U,\dotsc, T^{m+k+dn}x\in U,\]
which implies that $N(x,U)$ is an AP-set.
This ends the proof.
\end{proof}

\begin{prop}
Let $(X,T)$ be a dynamical system and $x\in X$.
Then the following conditions are equivalent:
\begin{enumerate}
\item $x$ is an $\AP$-recurrent point in $(X,T)$;
\item $x$ is an $\AP$-recurrent point in $(X,T^n)$ for some $n\in\N$;
\item $x$ is an $\AP$-recurrent point in $(X,T^n)$ for any $n\in\N$.
\end{enumerate}
\end{prop}
\begin{proof} The implications
(3)~$\Rightarrow$~(2)~$\Rightarrow$~(1) are clear. We only need to show (1)~$\Rightarrow$~(3).
Fix $n\in\N$. Without loss of generality, we can assume that $X=\overline{\Orb(x,T)}$. Then $(X,T)$ is topologically transitive system,
because $x$ is a recurrent point. Moreover, as $x$ is $\AP$-recurrent in $(X,T)$, applying Proposition~\ref{prop:AP-rec-vdW} we get that $(X,T)$ is a van der Waerden system. Denote $X_0=\overline{\Orb(x,T^n)}$.
It is well known (see~\cite[Lemma~6.5]{Li2012} for example) that
the interior of $X_0$ (with respect to the topology of $X$) is dense in $X_0$, that is, $X_0$ is regular closed subset of $X$.
By Lemma~\ref{lem:vdW-AP}, the collection of multi-recurrent points in $(X,T)$ is dense in $X$. By Lemma~\ref{lem:multi-rec-n}, every point multi-recurrent under action of $T$ is also multi-recurrent for $T^n$. Hence the set of multi-recurrent points
of $(X_0,T^n)$ is dense in $X_0$.
By Lemma~\ref{lem:vdW-AP} again, $(X_0,T^n)$ is a van der Waerden system.
By Proposition~\ref{prop:AP-rec-vdW} we obtain that every transitive point in $(X_0,T^n)$ is $\AP$-recurrent. So $x$ is also $\AP$-recurrent in $(X_0,T^n)$.
\end{proof}

In the proof of next result we will employ the technique developed in~\cite{Li2012} and show that
every $\AP$-recurrent point can be lifted through factor maps.

\begin{prop}\label{prop:AP-rec-lift}
Let $\pi\colon(X,T)\to(Y,S)$ be a factor map.
If $y\in Y$ is an $\AP$-recurrent point, then there exists an $\AP$-recurrent point $x\in X$
such that $\pi(x)=y$.
\end{prop}
\begin{proof}
It is clear that for any $n\in\Z$ and any $F\in\AP$,
the translation of $F$ by $n$ denoted by $n+F=\{n+k\in\N\colon k\in F\}$, is also an AP-set.
In other words, the family $\AP$ is translation invariant (see \cite[page 263]{Li2012}).
Recall that the family $\AP$ has the Ramsey property.
Then by \cite[Lemma~3.4]{Li2012}, all the assumptions of Proposition~4.5 in~\cite{Li2012} are satisfied by $\AP$.
The result follows by application of \cite[Proposition 4.5]{Li2012} to the family $\AP$.
\end{proof}

\begin{rem}
The proof of Proposition~\ref{prop:AP-rec-lift} which is short and compact,
uses advanced machinery from \cite{Li2012}.
 Another, more elementary proof will be given later in Section~\ref{sec:6}.
\end{rem}

To characterize when a transitive system is a van der Waerden system,
we need the following definition. It is a special case of a notion considered in~\cite{Li2011}.

\begin{defn}
We say that $x\in X$ is an \emph{$\AP$-transitive point} if $N(x,U)$ is an AP-set for every non-empty open set $U\subset X$.
\end{defn}

\begin{prop} \label{prop:vdW-sys-transitive}
Let $(X,T)$ be a transitive system. Then the following conditions are equivalent:
\begin{enumerate}
  \item $(X,T)$ is a van der Waerden system;
  \item there exists an $\AP$-transitive point;
  \item every transitive point is an $\AP$-transitive point.
\end{enumerate}
\end{prop}

\begin{proof}
The implication (3) $\Rightarrow$ (2) is obvious and (2)~$\Rightarrow$~(1) follows from Proposition~\ref{prop:AP-rec-vdW}.
We only need to show that (1)~$\Rightarrow$~(3).

Let $x$ be a transitive point. It follows from Proposition~\ref{prop:AP-rec-vdW} that $x$ is an $\AP$-recurrent point.
Fix a non-empty open subset $U$ of $X$.
There exist a neighborhood $V$ of $x$ and $k\in\N$ such that $T^k V\subset U$.
Then $k+N(x,V)\subset N(x,U)$.
But $N(x,V)$ is an AP-set and so also $N(x,U)$ is an AP-set, which proves that  $x$ is an $\AP$-transitive point.
\end{proof}

\begin{prop} \label{prop:vdW-system-product}
Let $(X,T)$ be a transitive system.
If $(X,T)$ is a van der Waerden system, then $(X^n,T^{(n)})$ is also a van der Waerden system for every $n\in\N$,
where $T^{(n)}$ denotes $n$-times Cartesian product $T^{(n)}=T\times T\times\dotsb\times T$.
\end{prop}
\begin{proof}
Let $U_1,U_2,\dotsc,U_n$ be non-empty open subsets in $X$.
Pick a transitive point $x\in U_1$.
Then there exist $k_1,k_2,\dotsc,k_{n-1}\in\N$ such that
$T^{k_1}x\in U_2, T^{k_2}x\in U_3,\dotsc, T^{k_{n-1}}x\in U_n$.
Since $(X,T)$ is a van der Waerden system, $x$ is $\AP$-recurrent.
This immediately implies that  $(x,T^{k_1}x,T^{k_2}x, \ldots, T^{k_{n-1}}x)$
is $\AP$-recurrent in $(X^n,T^{(n)})$, hence $(X^n,T^{(n)})$ has a dense set of $\AP$-recurrent points.
The proof is finished by application of Lemma \ref{lem:G-delta-AP}.
\end{proof}

The following example shows that Proposition~\ref{prop:vdW-system-product} is no longer true
if we do not assume that $(X,T)$ is transitive. As a byproduct, we obtain two
van der Waerden systems whose product is not a van der Waerden system.

\begin{exmp}
Let $n_1=2$ and define inductively $n_{k+1}=(n_{k})^3$. Put $A_k=[n_k, (n_k)^2]\cap \N$ and
denote $S=\bigcup_{k=1}^\infty A_{2k}$ and $R=\bigcup_{k=1}^\infty A_{2k+1}$. Clearly, $S\cap R=\emptyset$.
Denote by $X_S$ and $X_T$ the following subshifts  (so-called spacing shifts, see \cite{BNOST13}).
\begin{align*}
X_S&=\bigl\{x\in\{0,1\}^\N\colon x_i=x_j=1 \implies \vert i-j\vert\in S\cup\{0\}\bigr\}\\
X_R&=\bigl\{x\in\{0,2\}^\N\colon x_i=x_j=2 \implies \vert i-j\vert\in R\cup\{0\}\bigr\}.
\end{align*}

We can consider $X_S$ and $X_R$ as subshifts of $\{0,1,2\}^\N$.
Let $X=X_S\cup X_T \subset \{0,1,2\}^\N$. For a word $w$ over $\{0,1,2\}^\N$ we write $[w]_S=[w]\cap X_S$ and $[w]_R=[w]\cap X_R$.
First note that the product system $(X\times X,\sigma\times \sigma)$ is not a van der Waerden system.
This is because
\[
  N_{\sigma\times\sigma}(([1]\times [2])\cap X,([1]\times [2] )\cap X)= N_\sigma([1]_S,[1]_S)\cap N_\sigma([2]_R,[2]_R)=S\cap R=\emptyset.
\]
Now we show that $(X,\sigma)$ is a van der  Waerden system.
It it enough to prove that both $(X_S,\sigma)$ and $(X_R,\sigma)$ are van der  Waerden systems.
We will consider only the case of $(X_S,\sigma)$, since the proof for $(X_R,\sigma)$ is the same.

Fix a word $w\in \mathcal{L}(X_S)$,
take any positive integer $k$ such that $n_{2k}>2(d+|w|)$
and consider the following sequence $x= \bigl(w 0^{n_{2k}}\bigr)^{d+1} 0^\infty$.
We claim that $x\in X_S$.
Take any integers $i<j$ with $x_i=x_j=1$. If $j-i\leq |w|$, then $j-i \in S$ by the choice of $w$.
In the remaining case $j-i> |w|$ we have
\[
n_{2k}\leq j-i \leq (d+1)|w 0^{n_{2k}}|= (d+1)(|w|+n_{2k})\leq
\frac{n_{2k}}{2} \left(\frac{n_{2k}}{2}+n_{2k}\right)<(n_{2k})^2,
\]
therefore also in this case $j-i\in S$. Indeed, $x\in X_S$.
Put $m=|w 0^{n_{2k}}|$ and observe that $x, T^mx, T^{2m}x,\dotsc,T^{dm}x\in [w]_S$.
But for every nonempty open set $U\subset X_S$ we can find a word $w$ such that $[w]_S\subset U$ and then there is $m$ such that
\[x \in U\cap T^{-n}U\cap T^{-2n}U\cap \dotsb \cap T^{-dn}U.\]
This shows that $(X_S,\sigma)$ is a van der  Waerden system.
\end{exmp}

By Proposition \ref{prop:Delta-transitive} every $\Delta$-transitive system is  a van der Waerden system.
On the other hand, \cite[Proposition~3]{M10} provides an example of a strongly mixing system
which is not $\Delta$-transitive.
In fact, we will show that the example in~\cite[Proposition~3]{M10}   is not even a van der Waerden system.

\begin{prop}\label{prop:mixing-not-vdW}
There exists a strongly mixing system which is not a van der Waerden system.
\end{prop}
\begin{proof}
Let $\F$ be a collection of finite words over $\{0,1\}$ satisfying the following two conditions:
the word $11$ is in $\F$ and
if $u$ and $v$ are two finite words over $\{0,1\}$ such that $|u|=|v|$, then the word $1u1v1$ is in $\F$.
Let $X=X_\F$ be the subshift specified by taking $\F$ as the collection of forbidden words.
Note that $X$ is non-empty since $0^\infty,0^n10^\infty\in X$ for every $n\geq 0$.

Put $W=[1]_X$ and assume that there exists $n\in \N$ such that $W\cap \sigma^{-n}W\cap \sigma^{-2n}W\neq\emptyset$.
Then there exist two words $u$ and $v$ with length $n-1$ such that $1u1v10^\infty\in X$, which is a contradiction.
This shows that $(X,\sigma)$ is not a van der Waerden system.

Now we show that $(X,\sigma)$ is strongly mixing.
Let $u$ and $v$ be two  words in the language of $X$.
Put $N=|u|+|v|$. For every $n\geq N$, one has $u0^nv0^\infty\in X$.
This implies that $n\in N([u]_X,[v]_X)$ for every $n\geq N$, proving that $(X,\sigma)$ is strongly mixing.
\end{proof}

\begin{rem}
In fact, one can show that the only $\AP$-recurrent point of $(X,\sigma)$
in the Proposition \ref{prop:mixing-not-vdW} is the fixed point $0^\infty$.
\end{rem}

\begin{prop}\label{prop:lem_vdW_factor}
Let $\pi\colon (X,T)\to (Y,S)$ be a factor map.
\begin{enumerate}
  \item\label{prop:lem_vdW_factor:1} If $(X,T)$ is a van der Waerden system, then so is $(Y,S)$.
  \item\label{prop:lem_vdW_factor:2} If $(Y,S)$ is a van der Waerden system, then there exists a van der Waerden subsystem $(Z,T)$ of $(X,T)$
  such that $\pi(Z)=Y$.
  \item\label{prop:lem_vdW_factor:3} If $\pi$ is almost one to one, then $(X,T)$ is a van der Waerden system
  if and only if $(Y,S)$ is a van der Waerden system.
\end{enumerate}
\end{prop}
\begin{proof}
\eqref{prop:lem_vdW_factor:1}: It is a consequence of the definition of van der Waerden system.
\eqref{prop:lem_vdW_factor:2}:  By Lemma~\ref{lem:G-delta-AP} the set of $\AP$-recurrent point of $(Y,S)$, denoted by $Y_0$,
is a dense subset of $Y$. Then by Proposition~\ref{prop:AP-rec-lift},
for every $y\in Y_0$, there exists $x_y\in X$ such that $\pi(x_y)=y$ and $x_y$ is $\AP$-recurrent.
Let $X_0=\{x_y\colon y\in Y_0\}$ and $Z=\overline{\bigcup_{x\in X_0}\Orb(x,T)}$.
Clearly $\pi(Z)=Y$.
For every $x\in X_0$, any point in $\Orb(x,T)$ is $\AP$-recurrent.
So $Z$ has a dense set of $\AP$-recurrent points and so $(Z,T)$ is a van der Waerden system by Lemma~\ref{lem:G-delta-AP}.

\eqref{prop:lem_vdW_factor:3}: By \eqref{prop:lem_vdW_factor:1} we only need to prove that when $\pi$ is almost one-to-one and  $(Y,S)$ is a van der Waerden system then $(X,T)$
is also a van der Waerden system.

If we put $X_0=\{x\in X\colon \pi^{-1}(\pi(x))=\{x\}\}$, then by the definition of an almost one-to-one factor, $X_0$ is residual in $X$.
For every $x\in X_0$ and every neighborhood $U$ of $x$ there is a neighborhood $V$ of $\pi(x)$ such that $\pi^{-1}(V)\subset U$.
This implies that $\pi(X_0)$ is residual in $Y$.
By Lemma~\ref{lem:G-delta-AP}, the set of $\AP$-recurrent point of $(Y,S)$, denoted again by $Y_0$,
is a residual subset of $Y$.
Then $\pi(X_0)\cap Y_0$ is also residual in $Y$ and
$\pi^{-1}(\pi(X_0)\cap Y_0)$ is residual in $X$.
By Proposition~\ref{prop:AP-rec-lift}, every point in $\pi^{-1}(\pi(X_0)\cap Y_0)$ is $\AP$-recurrent.
Thus $(X,T)$ is a van der Waerden system by Lemma~\ref{lem:G-delta-AP}.
\end{proof}

\section{Multiple IP-recurrence property}
To get a dynamical characterization of C-sets,
the second author of this paper
introduced in~\cite{Li2012} a class of dynamical system satisfying the multiple IP-recurrence property.
In this section, we study this property and its relation to the van der Waerden systems.

\begin{defn}
We say that a dynamical system $(X,T)$ has the \emph{multiple IP-recurrence property}
if for every non-empty open subset $U$ of $X$, every $d\geq 1$ and
every IP-sets $FS\{p_i^{(1)}\}_{i=1}^\infty$,
$FS\{p_i^{(2)}\}_{i=1}^\infty ,\dotsc, FS\{p_i^{(d)}\}_{i=1}^\infty$ in $\N$,
there exists a finite subset $\alpha$ of $\N$
such that
\[U\cap T^{-\sum_{i\in\alpha}p_i^{(1)}}U\cap T^{-\sum_{i\in\alpha}p_i^{(2)}}U\cap\dotsb\cap T^{-\sum_{i\in\alpha}p_i^{(d)}}U\neq\emptyset.\]
\end{defn}

It is clear that if a dynamical system $(X,T)$ has the  multiple IP-recurrent property,
then it is a van der Waerden system.

By \cite[Theorem~A]{FK} we know that every E-system has the multiple IP-recurrent property.
It is shown in~\cite{GW93} that every E-system is either equicontinuous or sensitive.
We show that this dichotomy also holds for transitive systems with the multiple IP-recurrence property.
This is an extension of the main result in ~\cite{GW93} because there are transitive multiply
IP-recurrent systems which are not E-systems (see Remark~\ref{rem:5.6}).

\begin{thm}\label{thm:IPrecEqorSens}
If $(X,T)$ is a transitive system with the multiple IP-recurrence property,
then $(X,T)$ is either equicontinuous or sensitive.
\end{thm}
\begin{proof}
Every transitive system is either almost equicontinuous or sensitive (see \cite{AAB}),
so let us assume that  $(X,T)$ is almost equicontinuous.
It suffices to show that $(X,T)$ is minimal, since every minimal almost equicontinuous system is equicontinuous
(see \cite{AY80}). 

Pick a transitive point $x$ of $(X,T)$.
By \cite[Theorem 2.4]{AAB} the set of transitive points coincides with the set of equicontinuity points.
Then $x$ is also a equicontinuity point.
Fix any open neighborhood $U$ of $x$ and take $\eps>0$ such that the open $\eps$-ball around $x$ is contained in  $U$.
By equicontinuity of $x$ there is $\delta>0$ such that if $\rho(x,y)<\delta$ then $\rho(T^ix,T^iy)<\eps/2$
for every integer $i\geq 0$. Let $V$ denote the open $\delta$-ball around $x$.
Since $(X,T)$ has the multiple IP-recurrence property, for every IP-set $\FS\set{p_i}_{i=1}^\infty$ there exists a finite subset $\alpha$ of $\N$
such that $V\cap T^{-\sum_{i\in\alpha}p_i}V\neq\empty\emptyset$.
It follows that $N(V,V)$ is an IP$^*$-set. In particular, $N(V,V)$ is a syndetic set.
Next observe, that if $y\in V$,
then $\rho(x,y)<\delta$. Therefore if $y,T^ny\in V$, then $\rho(T^nx,T^ny)<\eps/2$ and $\rho(T^ny,x)<\eps/2$. It follows $T^nx\in U$ and
therefore $N(V,V)\subset N(x,U)$. So $N(x,U)$ is syndetic.
This implies that $x$ is a minimal point and hence $(X,T)$ is minimal.
\end{proof}

\begin{rem}
It is shown in~\cite{AAB} that there exists an almost equicontinuous system $(X,T)$ which is not equicontinuous.
By Proposition~\ref{rem:aeq_vdW} the system $(X,T)$ is a van der Waerden system.
But it can not have the multiple IP-recurrence property by Theorem~\ref{thm:IPrecEqorSens}.
\end{rem}

Next, we will modify the example constructed in Proposition~\ref{prop:mixing-not-vdW},
to obtain a strongly mixing van der Waerden system without the multiple IP-recurrence property.

\begin{prop}
There is a strongly mixing system which is a van der Waerden system
but does not have the multiple IP-recurrence property.
\end{prop}
\begin{proof}
We are going to construct a subshift $X$ and two IP-sets $\FS\{p_i\}_{i=1}^\infty, \FS\{q_i\}_{i=1}^\infty$
such that for every finite $\alpha \subset \N$ we have
\[[1]_X\cap T^{-\sum_{i\in\alpha}p_i}[1]_X\cap T^{-\sum_{i\in\alpha}q_i}[1]_X=\emptyset.\]
Let us take any sequences $\{p_i\}_{i=1}^\infty$ and  $\{q_i\}_{i=1}^\infty$ satisfying:
\[
\sum_{j=1}^n p_j < p_{n+1}\text{ and } q_n=2^np_{n+1} \text{ for every }n\in\N.
\]

Let $\F$ be a collection of finite words over $\{0,1\}$ satisfying the following two conditions:
the words $11$ is in $\F$, and
if $u$ and $v$ are two finite words over $\{0,1\}$ such that
 $|u|=\sum_{i\in\alpha} p_i-1$ and $|u|+|v|= \sum_{i\in\alpha} q_i-2$ for some finite subset $\alpha$ of $\N$
then the word $1u1v1$ is in $\F$.
Let $X$ be the subshift specified by taking $\F$ as the collection of forbidden words.
Note that $X$ is nonempty since $0^\infty\in X$.

Let $w'$ and $w''$ be two words in the language of $X$.
Take any $s$ such that
\[
|w'|+|w''|+2
< p_{s+1}<q_s<q_{s+1}.
\]
It follows that if $\alpha\subset\N$ is a finite set such that
\[
\sum_{i\in \alpha}p_i< |w'|
\]
then $\max\alpha\le s$.
Let $N=q_{s+1}$.
For any $n\geq N$, let $x_n=w'0^nw''0^\infty$.
We will show that $x_n$ is a point in $X$ and hence $X$ is a mixing subshift. We need to show that no word from $\F$ may appear in $x_n$.
First note that the word $11$  does not appear in $x_n$, since the word $11$  appears neither in $w'$ nor in $w''$.
Suppose that for some non-empty words $u$ and $v$ over $\{0,1\}$ the word $1u1v1$ appears in $x_n$. If it is a subblock of $w'$ or $w''$, then it does not belong to $\F$. Now assume that $1u1v1$ appears in $x_n$, but neither in $w'$, nor in $w''$. Therefore either $1u1$ is a subword of $w'$ or
$1v1$ is a subword of $w''$. In the first case, if $\alpha\subset\N$ is a finite set such that
\[
\quad \sum_{i\in \alpha}p_i=|u|+1\le |w'|<p_{s+1},
\]
then $\max\alpha\le s$, hence
\[
\sum_{i\in \alpha}q_i\le \sum_{j=1}^sq_j<q_{s+1}.
\]
But on the other hand $|v|\geq n\geq q_{s+1}$ and therefore $|u|+|v|+2> \sum_{i\in\alpha} q_i$.
It implies that $1u1v1\notin\F$.

In the second case note that $|w''|\ge |v|+2$. Now, if $\alpha\subset\N$ is a finite set such that
\[
\sum_{i\in \alpha}p_i=|u|+1\ge n \ge q_{s+1},
\]
then $\max\alpha> s$, hence
\[
\sum_{i\in \alpha}q_i\ge q_{s+1}\ge
\sum_{i\in \alpha}p_i +p_{s+1}>|u|+1+|w'|+|w''|> |u|+|v|+2.
\]
It implies that $1u1v1\notin\F$.
Hence $x_n\in X$ and therefore $n\in N([u]_X,[v]_X)$ and $(X,\sigma)$ is strongly mixing.

By a similar argument,
one can show that $(X,\sigma)$ is a van der Waerden system.

Finally observe that if
\[[1]_X\cap T^{-\sum_{i\in\alpha}p_i}[1]_X\cap T^{-\sum_{i\in\alpha}q_i}[1]_X\neq\emptyset\]
then there are two finite words $u,v$ such that $1u1v1$ is in the language of $X$
and $|u|=\sum_{i\in\alpha}p_i-1$ and $|u|+|v|+1=\sum_{i\in\alpha}q_i-1$.
This contradicts the definition of $X$.
Thus $(X,\sigma)$ does not have multiple IP-recurrence property.
\end{proof}

In the rest of this section, we show that there is a large family of subshifts,
with the multiple IP-recurrence property.
For a function $f\colon\Zp\to[0,\infty)$,
we define
\[\Psi_f=\Bigl\{x\in\{0,1\}^\N:
\forall p\in \Zp,\forall i\in \N, \sum_{r=i}^{i+p-1}x_r\leq f(p)\Bigr\}\]
and call it the \emph{bounded density subshift} generated by $f$.
Bounded density shifts were introduced by Stanley in \cite{S13}.
Stanley proved also that to define $\Psi_f$ we can consider only \emph{canonical functions} $f\colon\Zp\to[0,\infty)$.
By~\cite[Theorem~2.9]{S13} a function $f\colon\Zp\to[0,\infty)$ is
\emph{canonical} for the bounded density shift $\Psi_f$ if and only if:
\begin{enumerate}
  \item $f(0)=0$;
  \item $f(m+1)\in f(m)+\Zp$ for any $m\in\Zp$;
  \item $f(m+n)\leq f(m) +f(n)$ for any $n,m\in\Zp$.
\end{enumerate}
Note that if  $f(1)=0$, then $\Psi_f=\{0^\infty\}$.

\begin{thm}\label{thm:IPdensityShift}
If $f$ is an unbounded canonical function 
then the bounded density subshift $(\Phi_f,\sigma)$ generated by $f$ has the multiple IP-recurrent property.
\end{thm}

\begin{proof}
Fix a word $w$ in the language of $\Psi_f$ and let $U=[w]\cap \Psi_f$.
Take any $d\geq 1$ and any IP-sets $FS\{p_i^{(1)}\}_{i=1}^\infty$,
$FS\{p_i^{(2)}\}_{i=1}^\infty ,\dotsc, FS\{p_i^{(d)}\}_{i=1}^\infty$.
For simplicity of notation,  given a finite subset $\alpha$ of $\N$,
we define $p^{(i)}_\alpha=\sum_{j\in\alpha} p^{(i)}_j$.

Without loss of generality, we may assume that
for any $i\in\{1,\ldots,d\}$ and $j\in\N$ we have
\[p^{(i)}_j < p^{(i)}_{j+1}\text{and} p_j^{(i)}< p_j^{(i+1)} \;(\text{provided }i<d).
\]
Since $f$ is unbounded, there exists $p\in\N$ such that $f(p)>(d+1)|w|$ and $p\geq d|w|$.
There is $N\in\N$ such that if $\alpha\subset\N$
is a finite set with $\max \alpha\ge N$, then $\sum_{j\in\alpha} p^{(i)}_j>p+|w|$ for every $i\in\{1,\ldots,d\}$.
Note that for every $\alpha=\set{a_1,\ldots,a_s}\subset\N$ and any $1\le i <d$ we have
\[
p^{(i+1)}_\alpha\geq \sum_{j=1}^s p_{a_j}^{(i+1)}\geq \sum_{j=1}^s (p_{a_j}^{(i)}+1) \geq s+ p^{(i)}_\alpha.
\]
Denote $\beta=\set{N+1,\ldots,N+2p+1}$ and observe that  $p_{\beta}^{(i+1)}>p_{\beta}^{(i)}+2p$ for any $1\le i <d$
and $p_{\beta}^{(1)}>p+|w|$.
Let
\[
x=w 0^{p_\beta^{(1)}-|w|}w 0^{p_\beta^{(2)}-p_\beta^{(1)}-2|w|}w\ldots w 0^{p_\beta^{(d)}-p_\beta^{(d-1)}-d|w|}w 0^\infty.
\]
It is easy to see that $x\in\Psi_f$ and
\[
\sigma^{p_\beta^{(i)}}(x)\in [w] \text{ for }i=1,\ldots,d.
\]
Therefore
\[
U\cap \sigma^{-p_\beta^{(1)}}U\cap \sigma^{-p_\beta^{(2)}}U\cap\ldots\cap \sigma^{-p_\beta^{(d)}}U\neq\emptyset. \qedhere
\]
\end{proof}

\begin{rem}\label{rem:5.6}
By \cite[Theorem~2.14]{S13},
the bounded density shift $(\Phi_f,\sigma)$ in Theorem~\ref{thm:IPdensityShift}
is also strongly mixing.
If the function $f$ grows very slow, for example $f(n)=\log (n+1)$, then for any point $x\in \Phi_f$
one has
\[\lim_{n\to \infty}\frac{1}{n}\#(N(x,[1])\cap [1,n])\leq \lim_{n\to \infty}\frac{f(n)}{n}=0.\]
It follows that the only invariant measure of $(\Phi_f,\sigma)$ is the point mass on $\set{0^\infty}$.
But $\Psi_f$ is uncountable, hence $(\Psi_f,\sigma)$ is not an E-system.
Let $x$ be transitive point of $(\Phi_f,\sigma)$. By \cite[Theorems 8.5 and 4.4]{Li2012},
we know that $N(x,U)$ is a C-set for every  neighborhood $U$ of $x$.
Since  $(\Psi_f,\sigma)$ is not an E-system and $x$ is its transitive point,
there exists a neighborhood $V$ of $x$ such that $N(x,V)$ has the Banach density zero.
This gives a dynamical proof of a combinatorial result in~\cite{H09}
that there exists a C-set which has Banach density zero.
\end{rem}

\section{Multi-non-wandering points and van der Waerden center}\label{sec:6}

We say that a point $x\in X$ is  a \emph{non-wandering point}
if for every neighborhood $U$ of $x$  there exists an $n\in\N$ such that $U\cap T^{-n}U\neq\emptyset$.
Denote by $\Omega(X,T)$ the set of all non-wandering points of $(X,T)$.
It is easy to see that $\Omega(X,T)$ is a non-empty, closed and $T$-invariant.
So $(\Omega(X,T),T)$ also forms a dynamical system, so we
can consider non-wandering points of the subsystem $(\Omega(X,T),T)$.
To introduce the notion of Birkhoff center,
we define a (possibly transfinite) descending chain of non-empty closed and $T$-invariant subsets of $X$. We put inductively $\Omega_0(X,T)=X$, $\Omega_1(X,T)=\Omega(\Omega_0(X,T),T)$, and for every ordinal $\alpha$ we set $\Omega_{\alpha+1}(X,T)=\Omega(\Omega_\alpha(X,T),T)$.
We continue this process by a transfinite induction: if $\lambda$ is a limit ordinal we define
\[\Omega_\lambda(X,T)=\bigcap_{\alpha<\lambda}\Omega_\alpha(X,T).\]
In compact metric space decreasing family of closed sets is always at most countable, hence then there is a countable ordinal $\alpha$ such that
\[
X=\Omega_0(X,T) \supset \Omega_1(X,T) \supset \,\dotsb\, \supset \Omega_{\alpha}(X,T)=\Omega_{\alpha+1}(X,T)=\dotsb.
\]
We say that
$\Omega_\alpha(X,T)$ is the \emph{Birkhoff center} of
$(X,T)$ if $\Omega_{\alpha+1}(X,T)=\Omega_{\alpha}(X,T)$
and we define \emph{depth} of $(X,T)$ by
\[\mathrm{depth}(X,T)=\min\bigl\{ \alpha\colon \Omega_{\alpha}(X,T)=\Omega_{\alpha+1}(X,T) \bigr\}.\]
Note that compactness of $X$ implies that $\mathrm{depth}(X,T)<\omega_1$,
where $\omega_1$ is the first uncountable ordinal number.

Inspired by the notion of non-wandering points and the Birkhoff center,
we introduce multi-non-wandering points and the van der Waerden center.
\begin{defn}
Let $(X,T)$ be a dynamical system.
A point $x\in X$ is \emph{multi-non-wandering} if for every
open neighborhood $U$ of $x$ and every $d\in\N$ there exists an $n\in\N$ such that
\[U\cap T^{-n}U\cap T^{-2n}U\cap \dotsb \cap T^{-dn}U\neq\emptyset,\]
that is for every $d\in\N$, the diagonal $d$-tuple $(x,x,\dotsc,x)$ is non-wandering in
$(X^d,T\times T^2\times\dotsb\times T^d)$.
Denote by $\Omega^{(\infty)}(X,T)$ the collection of all multi-non-wandering points.
\end{defn}

First, we have the following characterization of multi-non-wandering points in a orbit closure of a point.
\begin{prop}\label{prop:multi-NW-AP}
Let $(X,T)$ be dynamical system and $x\in X$.
Suppose that $\overline{Orb(x,T)}=X$.
Then $y$ is a multi-non-wandering point if and only if $N(x,U)$ is an AP-set for
every neighborhood $U$ of $y$.
\end{prop}
\begin{proof}
First assume that $y$ is a multi-non-wandering point.
Fix a neighborhood $U$ of $y$.
For every $d\in\N$ there exists an $n\in\N$ such that the set
$V=U\cap T^{-n}U\cap T^{-2n}U\cap \dotsb \cap T^{-dn}U$
is non-empty and open. Since $\overline{\Orb(x,T)}=X$ there exists $m\geq 0$ such that $T^{m}x\in V\subset U$, and hence
\[T^{m+n}x\in U,\, T^{m+2n}x\in U,\,\ldots,\,T^{m+dn}x\in U,\]
that is $\{m+n,m+2n,\dotsc,m+dn\}\subset N(x,U)$. Thus $N(x,U)$ is an AP-set.

Fix a neighborhood $U$ of $y$ and assume that $N(x,U)$ is an AP-set.
There exist $m,n\in\N$ such that $\{m,m+n,m+2n,\dotsc,m+dn\}\in N(x,U)$.
Put $z=T^mx$. Then $z\in U\cap T^{-n}U\cap T^{-2n}U\cap \dotsb \cap T^{-dn}U$ and so $y$ is a multi-non-wandering point.
\end{proof}

The proof of following result is inspired by the set's forcing in~\cite{BF02} (consult \cite[Section 5]{Li2012} for more information on this topic).

\begin{thm} \label{thm:AP-multi-NW}
A set $F\subset\N$ is an AP-set if and only if
for every dynamical system $(X,T)$ and every $x\in X$, there is a multi-non-wandering point in $\overline{T^Fx}$,
where $T^Fx=\{T^nx\colon n\in F\}$.
\end{thm}
\begin{proof}
Assume that $F$ is an AP-set.
Let $(X,T)$ be a dynamical system and $x\in X$.
Without loss of generality, assume that $\overline{\Orb(x,T)}=X$.
Set $K=\overline{T^Fx}$.
Cover $K$ with closed balls with diameter less than $1$ and let $r_1$ be the cardinality of a finite subcover of this cover.
Then we can present
\[K=\bigcup_{i=1}^{r_1} K_{1,i},\]
where each $K_{1,i}$ is compact and has diameter less than $1$.
Since the family $\AP$ of AP-sets has the Ramsey property, there is an AP-set $F_1\subset F$ and $i_1$
such that $T^{F_1}x\in K_{1,i_1}$.
Set $K_1=K_{1,i_1}$.
Cover $K_1$ with closed balls with diameter less than $1/2$ and let $r_2$ be the cardinality of some finite subcover of this cover.
Write
\[K_1=\bigcup_{i=1}^{r_2} K_{2,i},\]
where each $K_{2,i}$ is compact and has diameter less than $1/2$.
By induction we have a sequence of compact sets $\{K_i\}_{i=1}^\infty$
and a sequence of AP-sets $\{F_i\}_{i=1}^\infty$
such that $K_{i+1}\subset K_i$, $\diam(K_i)<1/i$, $F_{i+1}\subset F_i$ and $T^{F_i}x\subset K_i$.
By the compactness of $X$, there is $y\in X$ such that $\bigcap_{i=1}^\infty K_i=\{y\}$.
For every neighborhood $U$ of $y$, there exists $i_0$ such that $K_{i_0}\subset U$.
Then $F_{i_0}\subset N(x,U)$, hence $N(x,U)$ is an AP-set.
Thus $y$ is a multi-non-wandering point by Proposition~\ref{prop:multi-NW-AP}.

Now assume that for every dynamical system $(X,T)$ and
every $x\in X$ there is a multi-non-wandering point in $\overline{T^Fx}$.
Let $x$ be the characteristic function of  $F$.
We can view $x$ as a point in the full shift $(\{0,1\}^{\Z_+},\sigma)$.
Put $X=\overline{\Orb(x,\sigma)}$ and
note that $N(x,[1]\cap X)=F$.
By assumption, there exists a multi-non-wandering point $y \in \overline{T^Fx} \subset [1]\cap X$.
By Proposition~\ref{prop:multi-NW-AP}, $F=N(x,[1]\cap X)$ is an AP-set, since $[1]\cap X$ is a neighborhood of $y$.
\end{proof}

\begin{thm} \label{thm:mwn-AP*}
Let $(X,T)$ be a dynamical system and $x\in X$ be such that $\overline{\Orb(x,T)}=X$.
Then
\begin{enumerate}
  \item If $U$ is a neighborhood of $\Omega^{(\infty)}(X,T)$ and $y\in X$, then $N(y,U)$ is an AP$^*$-set.
  \item If $M$ is a non-empty closed subset $X$ satisfying (1), then $\Omega^{(\infty)}(X,T)\subset M$, that is $\Omega^{(\infty)}(X,T)$ is characterized as the smallest subset of $X$ satisfying (1).
\end{enumerate}
\end{thm}
\begin{proof}
We first show that (1) holds. Take a neighborhood $U$ of $\Omega^{(\infty)}(X,T)$.
If there exists $z\in X$ such that $N(z,U)$ is not an AP$^*$-set,
then $F=N(z,U^c)$ is an AP-set.
By Theorem~\ref{thm:AP-multi-NW}, there exists a multi-non-wandering point in $\overline{T^Fz}\subset U^c$.
This contradicts $\Omega^{(\infty)}(X,T)\subset U$.

Assume that $M\subset X$ is non-empty, closed and satisfies (1).
We show that $\Omega^{(\infty)}(X,T)\subset M$.
Fix a multi-non-wandering point $z$.
Let $V$ be a neighborhood of $z$. It follows from Proposition~\ref{prop:multi-NW-AP} that $N(x,V)$ is an AP-set.
But $N(x,U)$ is an AP$^*$-set for every neighborhood $U$ of $M$. Hence $N(x,V)\cap N(x,U)\neq\emptyset$.
We get that $U\cap V\neq\emptyset$ for every neighborhood $V$ of $z$ and every neighborhood $U$ of $M$.
Thus $z\in M$, since $M$ is closed.
\end{proof}

Using the characterization of the set of multi-non-wandering points (Theorem \ref{thm:mwn-AP*}),
we can give another proof of Proposition~\ref{prop:AP-rec-lift} without using the advanced results on ultrafilters.
\begin{proof}[Another proof of Proposition ~\ref{prop:AP-rec-lift}]
Without loss of generality, assume that $Y=\overline{\Orb(y,S)}$.
Let
\[\mathcal{A}=\{A\subset X\colon (A,T) \text{ is a subsystem of }(X,T) \text{ and }Y\subset \pi(A)\}. \]
It is clear that $\mathcal{A}$ is not empty since $X\in \mathcal{A}$.
By the Zorn Lemma, there is a minimal (under the inclusion) element $Z\in\mathcal{A}$.
Pick $x\in\pi^{-1}(y)\cap Z$.
Note that $\overline{\Orb(x,T)}\subset Z$ and $Y\subset \pi(\overline{\Orb(x,T)})$.
By the minimality of $Z$, we have $Z=\overline{\Orb(x,T)}$.
Fix a neighbourhood $U$ of $\Omega^{(\infty)}(Z,T)$ and a neighborhood $V$ of $y$.
By Theorem \ref{thm:mwn-AP*}, $N(z,U)$ is an AP$^*$-set.
But $N(x,V)$ is an AP-set. Then there exists $n\in\N$ such that $T^nz\in U$ and $T^ny\in V$.
Thus $y\in\pi(\Omega^{(\infty)}(Z,T))$.
By the minimality of $Z$ again, one has $Z=\Omega^{(\infty)}(Z,T)$.
Thus $Z$ is a van der Waerden system and $x$ is $\AP$-recurrent by Proposition~\ref{prop:multi-NW-AP} and Lemma~\ref{lem:G-delta-AP}.
\end{proof}

It is clear that $\Omega^{(\infty)}(X,T)$ is closed and $T$-invariant.
So $(\Omega^{(\infty)}(X,T),T)$  also forms a dynamical system.
We can consider multi-non-wandering points in $(\Omega^{(\infty)}(X,T),T)$.
It is shown in Example~\ref{exmp:Omega-infty-neq} that
$\Omega^{(\infty)}\bigl(\Omega^{(\infty)}(X,T),T\bigr)$ may not equal to $\Omega^{(\infty)}(X,T)$.
Similar to the Birkhoff center, we introduce the van der Waerden center.
We put $\Omega^{(\infty)}_0(X,T)=X$, $\Omega^{(\infty)}_1(X,T)=\Omega^{(\infty)}(\Omega^{(\infty)}_0(X,T),T)$ and $\Omega^{(\infty)}_2(X,T)=\Omega^{(\infty)}(\Omega^{(\infty)}_1(X,T),T)$.
We continue this process.
Then $X=\Omega^{(\infty)}_0(X,T) \supset \Omega^{(\infty)}_1(X,T) \supset \dotsb$, $\Omega^{(\infty)}_{\alpha+1}(X,T)=\Omega^{(\infty)}(\Omega^{(\infty)}_\alpha(X,T),T)$,
$\Omega^{(\infty)}_\lambda(X,T)=\bigcap_{\alpha<\lambda}\Omega^{(\infty)}_\alpha(X,T)$,
where $\lambda$ is a limit ordinal number.
We say that
$\Omega^{(\infty)}_\alpha(X,T)$ is the \emph{van der Waerden center} of
$(X,T)$ if $\Omega^{(\infty)}_{\alpha+1}(X,T)=\Omega^{(\infty)}_{\alpha}(X,T)$.

Note that a dynamical system is a van der Waerden system if and only if
every point is multi-non-wandering.
The following result shows that the van der Waerden center is just the the maximal van der Waerden subsystem.

\begin{prop}\label{prop:vdW-center}
Let $(X,T)$ be a dynamical system and $\Omega^{(\infty)}_\alpha(X,T)$ be the van der Waerden center of $(X,T)$.
Then $\Omega^{(\infty)}_\alpha(X,T)$ is the closure of the set of $\AP$-recurrent points of $(X,T)$.
Furthermore, $(\Omega^{(\infty)}_\alpha(X,T),T)$ is the maximal van der Waerden subsystem of $(X,T)$.
\end{prop}

\begin{proof}
Let $Z$ be the set of $\AP$-recurrent points of $(X,T)$.
It is not hard to see that $Z\subset \Omega^{(\infty)}_{\gamma}(X,T)$ for every ordinal number $\gamma$.
So $\overline{Z}\subset \Omega^{(\infty)}_{\alpha}(X,T)$.

Since $\Omega^{(\infty)}_{\alpha+1}(X,T)=\Omega^{(\infty)}_{\alpha}(X,T)$, every point in the dynamical system
$(\Omega^{(\infty)}_\alpha(X,T),T)$ is multi-non-wandering, and then
$(\Omega_\alpha^{(\infty)}(X,T),T)$ is a van der Waerden system.
By Lemma~\ref{lem:vdW-AP}, the set of $\AP$-recurrent
points of $(\Omega^{(\infty)}_\alpha(X,T),T)$ is dense in $\Omega^{(\infty)}_{\alpha}(X,T)$.
Then $\Omega^{(\infty)}_{\alpha}(X,T)\subset \overline{Z}$.
\end{proof}

\begin{prop}
Let $\pi\colon (X,T)\to (Y,S)$ be a factor map.
Then the image of van der Waerden center of $(X,T)$ under $\pi$ coincides with the van der Waerden center of $(Y,S)$.
\end{prop}
\begin{proof}
Let $X_0$ and $Y_0$ be the set of all $\AP$-recurrent points in $(X,T)$ and $(Y,T)$ respectively.
By Proposition~\ref{prop:AP-rec-lift}, we have $\pi(X_0)=Y_0$.
Then the result follows from Proposition~\ref{prop:vdW-center}.
\end{proof}

\begin{exmp} \label{exmp:Omega-infty-neq}
There exists a dynamical system $(X,T)$ such that $\Omega^{(\infty)}\bigl(\Omega^{(\infty)}(X,T),T\bigr)\neq\Omega^{(\infty)}(X,T)$.

Take any increasing sequence $\set{z_n}_{n\in \Z}\subset (0,1)$ such that
$\lim_{n\to -\infty}z_n=0$ and $\lim_{n\to \infty}z_n=1$.
Let $X= \set{0,1}\cup \set{z_n : n\in \Z} \pmod 1$, that is, we view $z_n$ as a sequence on the unit circle.
Then we have $\lim_{n\to \infty} \rho(z_{-n},z_n)=0$, where $\rho$ is the standard metric on the unit circle.

Define
\[
Y=X\times \set{0}\cup \bigcup_{n=1}^\infty \bigcup_{j=2^n}^{2^{n+1}} \bigcup_{i=-n}^n \set{(z_i, 4^{-n}-j 2^{-n-1}4^{-n-1})}
\cup \bigcup_{n=0}^\infty{(z_{-n},2)}\cup (0,2).
\]
Clearly, if $j\neq s$ then $4^{-n}-j 2^{-n-1}4^{-n-1}\neq 4^{-n}-s 2^{-n-1}4^{-n-1}$ and $4^{-n}-4^{-n-1}>4^{-n-1}$.
Therefore the coordinates like $(z_i, 4^{-n}-j 2^{-n-1}4^{-n-1})$ uniquely determine a point in $Y$.
The set $Y$ is  a closed subset of a product space
$X\times [0,4]$. Therefore $Y$ with the maximum metric is compact.

Let $g(z_n)=z_{n+1}$ for every $n\in \Z$ and $g(0)=0\in X$.
For any integer $j\in [2^n,2^{n+1}]$ denote $a_j=(z_{-n}, 4^{-n}-j 2^{-n-1}4^{-n-1})$ and $b_j=(z_{n}, 4^{-n}-j 2^{-n-1}4^{-n-1})$.
Then we define a function $f\colon Y \to Y$ by putting
\[
f(x,y)=\begin{cases}
(g(x),y) & y=0 \text{ or } (y=2 \text{ and } x\neq z_0),\\
a_1 & y=2 \text{ and } x=z_0,\\
(g(x),y) & y\in (0,2) \text{ and } (x,y)\neq b_j \text{ for every }j,\\
a_{j+1} & (x,y)= b_j.
\end{cases}
\]
Clearly $f$ is a bijection and it is also not hard to verify that it is a homeomorphism.
Observe that $\Omega(f)=\set{(0,2)}\cup X\times \set{0}$. We are going to show that $\Omega^{(\infty)}(f)=\Omega(f)$.
Clearly both fixed points are in $\Omega^{(\infty)}(f)$. Now let us take any $m\in \Z$ and any open set $U\ni (z_m,0)$.
There is $N>0$ such that $(z_m,y)\in U$ for every $y\leq 4^{-N}$. Fix any $d>0$ and take $n> \max \{d,N, |m|\}$.
Now if we take any $j=2^n,\ldots, 2^{n}+d < 2^{n+1}-1$ then
\[
p_j=(z_m, 4^{-n}-j 2^{-n-1}4^{-n-1})\in Y \cap U.
\]
By the definition of $f$, for $j=0,\ldots, d-1$ we have $f^{2n+1}(p_j)=p_{j+1}$. In other words
\[
p_d\in U \cap f^{-2n-1}(U)\cap \dotsb \cap f^{-(2n+1)d}(U)\neq \emptyset.
\]
Indeed $(z_m,0)\in \Omega^{(\infty)}(f)$.
But
\[
\Omega^{(\infty)}(f|_{\Omega^{(\infty)}(f)})=\Omega(f|_{\Omega^{(\infty)}(f)})=\set{(0,0),(0,2)}.
\]
It follows that the van der Waerden center can be a proper subset of $\Omega^{(\infty)}(f)$.
\end{exmp}

\begin{rem}
It is shown in~\cite{N78} that if $\alpha$ is a countable ordinal, then
there exists a dynamical system $(X,T)$ with $\mathrm{depth}(X,T)=\alpha$.
We define the \emph{van der Waerden depth} of $(X,T)$ as
\[\mathrm{depth}^{(\infty)}(X,T)=\min\bigl\{ \alpha\colon \Omega^{(\infty)}_{\alpha+1}(X,T)=\Omega^{(\infty)}_{\alpha}(X,T) \bigr\}.\]
We conjecture that the van der Waerden depth is a countable ordinal and for every countable ordinal number $\alpha$
there exists a dynamical system $(X,T)$ such that $\mathrm{depth}^{(\infty)}(X,T)=\alpha$.
\footnote{Li and Zhang \cite{LiJ2015} gave a positive answer to this conjecture. }
\end{rem}

\section*{Acknowledgement}
A substantial part of this paper was written when the authors
attended the Activity ``Dynamics and Numbers", June--July 2014, held at
the Max Planck Institute  f\"ur Mathematik (MPIM) in Bonn, Germany.
Some part of the work was continued when the authors attended a conference held at
the Wuhan Institute of Physics and Mathematics,
China, and the satellite conference of 2014 ICM at
Chungnam National University, South Korea. We are grateful to the organizers for their hospitality.

D.~Kwietniak was supported by the National Science Centre (NCN)
under grant no. DEC-2012/07/E/ST1/00185.
J.~Li was supported in part by NSF of China (11401362, 11471125, 11326135)
and Shantou University  Scientific Research Foundation for Talents (NTF12021),
P. Oprocha was supported in part by Project of LQ1602 IT4Innovations Excellence in Science,
X.~Ye was supported by NSF of China (11371339,11431012).



\begin{thebibliography}{99}

\bibitem{AAB} Akin E, Auslander J, Berg K. When is a transitive map chaotic?
in: Convergence in ergodic theory and probability, de Gruyter, Berlin, 1996, 25--40.

\bibitem{AAG} Akin E, Auslander J, Glasner E. The topological dynamics of Ellis actions.
Mem. Amer. Math. Soc., vol. 195, no. 913, 2008.

\bibitem{Aus} Auslander J. Minimal flows and their extensions.
North-Holland Mathematics Studies, North-Holland Publishing Co., Amsterdam, 1988.

\bibitem{AY80} Auslander J, Yorke J. Interval maps, factors of maps, and chaos.
Tohoku Math. J.,1980, 32: 177--188.

\bibitem{BNOST13} Banks J, Nguyen T, Oprocha P, et al.
Dynamics of spacing shifts. Discrete Contin. Dyn. Syst., 2013, 33: 4207--4232.


\bibitem{BHR} Blanchard F, Host B, Ruette S. Asymptotic pairs in positive-entropy systems.
Ergod. Th. \& Dynam. Sys., 2002, 22: 671--686.

\bibitem{BF02} Blokh A, Fieldsteel A.
Sets that force recurrence, Proc. Amer. Math. Soc.
2002, 130: 3571--3578.

\bibitem{DHS08}De D, Hindman N, Strauss D.
A new and stronger central sets theorem,
Fund. Math., 2008, 199: 155--175.

\bibitem{Dow} Downarowicz T.
Survey of odometers and Toeplitz flows. Algebraic and topological dynamics,
Contemp. Math., 385, Amer. Math. Soc., Providence, RI, 2005, 7--37.

\bibitem{FHLO16} Fory\'s M, Huang W, Li J, et al.
Invariant scrambled sets, uniform rigidity and weak mixing.
 Israel J. Math., 2016, 211: 447--472.

\bibitem{F77} Furstenberg H.
Ergodic behavior of diagonal measures and a theorem of Szemerddi on arithmetic progressions.
 J. Analyse Math., 1977, 31: 204-256.

\bibitem{F81} Furstenberg H. Recurrence in Ergodic Theory and Combinatorial Number Theory.
Princeton Univ. Press, Princeton, NJ, 1981.

\bibitem{FK} Furstenberg H, and Katznelson Y.
An ergodic Szemer\'{e}di theorem for IP-systems and combinatorial theory.
J. Analyse Math., 1985, 45: 117--168.

\bibitem{FW78} Furstenberg H, Weiss B.
Topological dynamics and combinatorial number theory. J. Analyse Math.,  1978, 34: 61--85.

\bibitem{G} Glasner E. Ergodic theory via joinings.
Math. Surveys and Mono., 101, Amer. Math. Soc., 2003.

\bibitem{G94} Glasner E. Topological ergodic decompositions and applications to products
of powers of a minimal transformation.  J. Anal. Math., 1994, 64: 241--262.

\bibitem{GM89} Glasner S, Maon D. Rigidity in topological dynamics.
Ergod. Th. \& Dynam. Sys., 1989, 9: 309--320.

\bibitem{GW93} Glasner E, Weiss B. Sensitive dependence on initial conditions.
Nonlinearity, 1993, 6: 1067--1075.

\bibitem{H09} Hindman N. Small sets satisfying the central sets theorem. in:
Combinatorial Number Theory, B. Landman et al. (eds.), de Gruyter, Berlin, 2009, 57--63.


\bibitem{Kra15} Host B, Kra B, Maass A. Variations on topological recurrence.
Monatsh. Math., 2016, 179: 57--89

\bibitem{HPY07} Huang W, Park K, Ye X. Dynamical systems disjoint from all minimal
systems with zero entropy. Bull. Soc. Math. France, 2007, 135: 259--282.

\bibitem{HSY2014} Huang W, Shao S, Ye X. Pointwise convergence of multiple ergodic averages
and strictly ergodic models. arXiv:1406.5930[math.DS].

\bibitem{HYZ14} Huang W, Ye X, Zhang G. Lowering topological entropy over subsets revisted.
Trans. Amer. Math. Soc., 2014,  366: 4423--4442.

\bibitem{J11} Johnson J. A dynamical characterization of C sets. arXiv:1112.0715[math.DS].


\bibitem{L87} Lehrer E. Topological mixing and uniquely ergodic systems, Israel J. Math.,
1987, 57: 239--255.

\bibitem{Li2011} Li J. Transitive points via Furstenberg family, Topology Appl.
2011, 158: 2221--2231.

\bibitem{Li2012} Li J. Dynamical charaterzation of C-sets and its applications.
Fund. Math., 2012, 216: 259--286.

\bibitem{LiJ2015} Li J, Zhang R. Levels of Generalized Expansiveness.
J. Dyn. Diff. Equat., published online, DOI: 10.1007/s10884-015-9502-6.

\bibitem{LM} Lind D, Marcus B. An introduction to symbolic dynamics and coding.
Cambridge University Press, Cambridge, 1995.

\bibitem{M10} Moothathu T. Diagonal points having dense orbit.
 Colloq. Math., 2010, 120: 127--138.

\bibitem{N78} Neumann D. Central sequences in dynamical systems.
Amer. J. Math., 1978, 100: 1--18.

\bibitem{R33} Rado R. Studien zur Kombinatorik, Math. Zeit., 1933, 36: 242--280.

\bibitem{S13} Stanley B.  Bounded density shifts,
 Ergod. Th. \& Dynam. Sys., 2013, 33: 1891--1928.

\bibitem{S75} Szemer\'edi E.
On sets of integers containing no $k$ elements in arithmetic progression.
Acta Arith., 1975, 27: 199--245.

\bibitem{W27} Van der Waerden L. Beweis eine Baudetschen Vermutung Nieus Arch.
Wisk, 1927. 15:212--216.

\bibitem{W98} Weiss B.  Multiple recurrence and doubly minimal systems.
Topological dynamics and applications (Minneapolis, MN, 1995),
Contemp. Math., vol. 215, Amer. Math. Soc., Providence, RI, 1998, 189--196.

\end{thebibliography}
\end{document}